\documentclass[oneside,11pt,reqno]{amsart}

\usepackage[a4paper,margin=30mm]{geometry}
\usepackage{epsfig,color}

\usepackage{verbatim,upref,amsxtra,amssymb,amscd}
\usepackage{epstopdf}

\usepackage{varioref}

\DeclareMathAlphabet\mathscr{U}{eus}{m}{n} \SetMathAlphabet\mathscr{bold}{U}{eus}{b}{n} \DeclareMathAlphabet\matheur{U}{eur}{m}{n} \SetMathAlphabet\matheur{bold}{U}{eur}{b}{n}

\numberwithin{equation}{section}

\newtheorem{theo}{Theorem}[section]
\newtheorem{prop}[theo]{Proposition}
\newtheorem{lemm}[theo]{Lemma}
\newtheorem{coro}[theo]{Corollary}

\theoremstyle{definition}

\newtheorem{defi}[theo]{Definition}

\theoremstyle{remark}

\newtheorem{rema}[theo]{Remark}
\newtheorem{rems}[theo]{Remarks}

\newcommand{\defeq}{\,\raisebox{.08ex}{$\colon$}\hspace{-1.5ex}=}

\newcommand{\eqdef}{=\hspace{-1ex}\raisebox{.08ex}{$\colon$}}

\renewcommand{\k}{\mathbf{k}} \newcommand{\n}{\mathbf{n}}  \newcommand{\Zd}{\mathbb Z^d} \newcommand{\e}{\mathbf{e}} \newcommand{\Z}{\mathbb{Z}}

\begin{document}
\allowdisplaybreaks\frenchspacing

\setlength{\baselineskip}{1.2\baselineskip}

\title{Abelian Sandpiles and the Harmonic Model}

\author{Klaus Schmidt}

\address{Klaus Schmidt: Mathematics Institute, University of Vienna, Nordberg\-stra{\ss}e 15, A-1090 Vienna, Austria \newline\indent \textup{and} \newline\indent Erwin Schr\"odinger Institute for Mathematical Physics, Boltzmanngasse~9, A-1090 Vienna, Austria} \email{klaus.schmidt@univie.ac.at}

\author{Evgeny Verbitskiy}

\address{Evgeny Verbitskiy: Philips Research, High Tech Campus 36 (M/S 2), 5656 AE, Eindhoven, The Netherlands \newline\indent \textup{and} \newline\indent Department of Mathematics, University of Groningen, PO Box 407, 9700 AK, Groningen, The Netherlands} \email{evgeny.verbitskiy@philips.com}

\thanks{E.V. would like to acknowledge the hospitality
of the Erwin Schr\"{o}dinger Institute (Vienna), where
part of this work was done. E.V. is also grateful to Frank Redig, Marius van der Put and Tomas Tsang for illuminating discussions.
\\
\indent K.S. would like to thank EURANDOM (Eindhoven) and MSRI (Berkeley), for hospitality and support during part of this work.}




	\begin{abstract}
We present a construction of an entropy-preserving equivariant surjective map from the $d$-dimensional critical sandpile model to a certain closed, shift-invariant subgroup of $\mathbb{T}^{\mathbb{Z}^d}$ (the `harmonic model'). A similar map is constructed for the dissipative abelian sandpile model and is used to prove uniqueness and the Bernoulli property of the measure of maximal entropy for that model.
	\end{abstract}

\maketitle

\setcounter{tocdepth}{2} \tableofcontents

\section{\label{s:intro}Introduction}

For any integer $d\ge 2$ let
	\begin{equation}
	\label{eq:fundentropy}
h_d=\int_{0}^1 \cdots\int_{0}^1 \log \biggl( 2d-2\sum_{i=1}^d\cos(2\pi x_i)\biggr) \ dx_1\cdots dx_d,
	\end{equation}
$h_2=1.166$, $h_3=1.673$, etc. It turns out that for $d\ge 2$, $h_d$ is the topological entropy of three different $d$-dimensional models in mathematical physics, probability theory, and dynamical systems. For $d=2$, there is even a fourth model with the same entropy $h_d$.

\subsection{Four models} 

The $d$-dimensional \textit{\textbf{abelian sandpile model}} was introduced by Bak, Tang and Wiesenfeld in \cite{BTW1,BTW2} and attracted a lot of attention after the discovery of the Abelian property by Dhar in \cite{Dhar0}. The set of infinite allowed configurations of the sandpile model is the shift-invariant subset $\mathcal{R}_\infty \subset \{0,\ldots,2d-1\}^{\mathbb{Z}^d}$ defined in \eqref{eq:sandpiles} and discussed in Section \ref{s:sandpiles}.\footnote{In the physics literature it is more customary to view the sandpile model as a subset of $\{1,\dots ,d\}^{\mathbb{Z}^d}$ by adding $1$ to each coordinate.} In \cite{Dhar2}, Dhar showed that the topological entropy of the shift-action $\sigma _{\mathcal{R}_\infty }$ on $\mathcal{R}_\infty $ is also given by \eqref{eq:entropy}, which implies that every shift-invariant measure $\mu $ of maximal entropy on $\mathcal{R}_\infty $ has entropy \eqref{eq:fundentropy}. Shift-invariant measures on $\mathcal{R}_\infty $ were studied in some detail by Athreya and Jarai in \cite{Jarai1,Jarai2}, Jarai and Redig in \cite{JR}; however, the question of uniqueness of the measure of maximal entropy is still unresolved.

\textit{\textbf{Spanning trees}} of finite graphs are classical objects in combinatorics and graph theory. In 1991, Pemantle in his seminal paper \cite{Pemantle} addressed the question of constructing \textit{uniform} probability measures on the set $\mathcal{T}_d$ of infinite spanning trees on $\mathbb{Z}^d$ --- i.e., on the set of spanning subgraphs of $\mathbb{Z}^d$ without loops. This work was continued in 1993 by Burton and Pemantle \cite{BP}, where the authors observed that the topological entropy of the set of all spanning trees in $\mathbb{Z}^d$ is also given by the formula \eqref{eq:fundentropy}. Another problem discussed in \cite{BP} is the uniqueness of the shift-invariant measure of maximal entropy on $\mathcal{T}_d$ (the proof in \cite{BP} is not complete, but Sheffield has recently completed the proof in \cite{Sheffield}.

This coincidence of entropies raised the question about the relation between these models. A partial answer to this question was given in 1998 by R.~Solomyak in \cite{Solomyak}: she constructed injective mappings from the set of rooted spanning trees on finite regions of $\mathbb{Z}^d$ into $X_{f^{(d)}}$ such that the images are sufficiently separated. In particular, this provided a direct proof of coincidence of the topological entropies of $\alpha _{f^{(d)}}$ and $\sigma _{\mathcal{T}_d}$ without making use of formula (\ref{eq:fundentropy}).

In dimension 2, spanning trees are related not only to the sandpile models (cf. e.g., \cite{Redig} for a detailed account) and, by \cite{Solomyak}, to the harmonic model, but also to a \textit{\textbf{dimer model}} (more precisely, to the \textit{even} shift-action on the two-dimensional dimer model) by \cite{BP}.

However, the connections between the abelian sandpiles and spanning trees (as well as dimers in dimension 2), are \textit{non-local}: they are obtained by restricting the models to finite regions in $\mathbb{Z}^d$ (or $\mathbb{Z}^2$) and constructing maps between these restrictions, but these maps are not consistent as the finite regions increase to $\mathbb{Z}^d$.

In this paper we study the relation between the infinite abelian sandpile models and the algebraic dynamical systems called the \textit{\textbf{harmonic models}}. The purpose of this paper is to define a shift-equivariant, surjective \textit{local} mapping between these models: from the infinite critical sandpile model $\mathcal{R}_\infty $ to the harmonic model. Although we are not able to prove that this mapping is almost one-to-one it has the property that it sends every shift-invariant measure of maximal entropy on $\mathcal{R}_\infty $ to Haar measure on $X_{f^{(d)}}$. Moreover, it sheds some light on the somewhat elusive group structure of $\mathcal{R}_\infty$.

Firstly, the dual group of $X_{f^{(d)}}$ is the group
	$$
\mathcal{G}_d= R_d/(f^{(d)}),
	$$
where $R_d=\mathbb Z[u_1^\pm, \ldots, u_d^{\pm}]$ is the ring of Laurent polynomials with integer coefficients in the variables $u_1,\ldots, u_d$, and $(f^{(d)})$ is the principal ideal in $R_d$ generated by $f^{(d)}=2d-\sum_{i=1}^d (u_i+u_i^{-1})$. The group $\mathcal{G}_d$ is the correct infinite analogue of the groups of \textbf{\textit{addition operators}} defined on finite volumes, see \cite{Dhar1, Redig} (cf. Section \ref{s:conclusions}).

Secondly, the map $\xi _{I_d}$ constructed in this paper gives rise to an equivalence relation $\sim$ on $\mathcal{R}_\infty$ with
	$$
x\sim y\iff x-y\in\ker(\xi _{I_d}),
	$$
such that $\mathcal{R}_\infty /_\sim$ is a compact abelian group. Moreover, $\mathcal{R}_\infty /_\sim$, viewed as a dynamical system under the natural shift-action of $\mathbb{Z}^d$, has the topological entropy \eqref{eq:fundentropy}. This extends the result of \cite{MRS}, obtained in the case of dissipative sandpile model, to the critical sandpile model.

Finally, we also identify an algebraic dynamical system isomorphic to the dissipative sandpile model. This allows an easy extension of the results in \cite{MRS}: namely, the uniqueness of the measure of maximal entropy on the set of infinite recurrent configurations in the dissipative case. Unfortunately, we are not yet able to establish the analogous uniqueness result in the critical case.

\subsection{Outline of the paper} Section \ref{s:green} investigates certain multipliers of the potential function (or Green's function) of the elementary random walk on $\mathbb{Z}^d$. In Section \ref{s:harmonic} these results are used to describe the \textit{homoclinic points} of the harmonic model. These points are then used to define shift-equivariant maps from the space $\ell ^\infty (\mathbb{Z}^d,\mathbb{Z})$ of all bounded $d$-parameter sequences of integers to $X_{f^{(d)}}$. In Section \ref{s:sandpiles} we introduce the critical and dissipative sandpile models. In Section \ref{s:critical} we show that the maps found in Section \ref{s:harmonic} send the critical sandpile model $\mathcal{R}_\infty $ \textit{onto} $X_{f^{(d)}}$, preserve topological entropy, and map every measure of maximal entropy on $\mathcal{R}_\infty $ to Haar measure on the harmonic model. After a brief discussion of further properties of these maps in Subsection \ref{ss:properties}, we turn to dissipative sandpile models in Section \ref{s:dissipative} and define an analogous map to another closed, shift-invariant subgroup of $\mathbb{T}^{\mathbb{Z}^d}$. The main result in \cite{MRS} shows that this map is almost one-to-one, which implies that the measure of maximal entropy on the dissipative sandpile model is unique and Bernoulli.

\section{A potential function and its $\ell ^1$-multipliers
	\label{s:green}
}

Let $d\ge1$. For every $i=1,\dots ,d$ we write $\mathbf{e}^{(i)}=(0,\dots ,0,1,0,\dots ,0)$ for the $i$-th unit vector in $\mathbb{Z}^d$, and we set $\mathbf{0}=(0,\dots ,0)\in\mathbb{Z}^d$.

We identify the cartesian product $W_d=\mathbb{R}^{\mathbb{Z}^d}$ with the set of formal real power series in the variables $u_1^{\pm1},\dots ,u_d^{\pm1}$ by viewing each $w=(w_\mathbf{n})\in W_d$ as the power series
	\begin{equation}
	\label{eq:powerseries}
\smash[t]{\sum_{\mathbf{n}\in\mathbb{Z}^d}w_\mathbf{n}u^\mathbf{n}}
	\end{equation}
with $w_\mathbf{n}\in\mathbb{R}$ and $u^{\mathbf{n}}=u_1^{n_2}\cdots u_d^{n_d}$ for every $\mathbf{n}=(n_1,\dots ,n_d)\in\mathbb{Z}^d$. The \textit{involution} $w\mapsto w^*$ on $W_d$ is defined by
	\begin{equation}
	\label{eq:involution}
w^*_\mathbf{n}=w_{-\mathbf{n}},\enspace \mathbf{n}\in\mathbb{Z}^d.
	\end{equation}

For $E\subset \mathbb{Z}^d$ we denote by $\pi _E\colon W_d\longrightarrow \mathbb{R}^E$ the projection onto the coordinates in $E$.

\medskip For every $p\ge1$ we regard $\ell ^p(\mathbb{Z}^d)$ as the set of all $w\in W_d$ with
	$$
\|w\|_p=\biggl(\sum_{\mathbf{n}\in\mathbb{Z}^d}|w_\mathbf{n}|^p \biggr)^{1/p}<\infty .
	$$
Similarly we view $\ell ^\infty (\mathbb{Z}^d)$ as the set of all bounded elements in $W_d$, equipped with the supremum norm $\|\cdot \|_\infty $. Finally we denote by $R_d=\mathbb{Z}[u_1^{\pm1},\dots ,u_d^{\pm1}]\subset \ell ^1(\mathbb{Z}^d)\subset W_d$ the ring of Laurent polynomials with integer coefficients. Every $h$ in any of these spaces will be written as $h=(h_\mathbf{n})= \sum_{\mathbf{n}\in\mathbb{Z}^d}h_\mathbf{n} u^{\mathbf{n}}$ with $h_\mathbf{n}\in\mathbb{R}$ (resp. $h_\mathbf{n}\in\mathbb{Z}$ for $h\in R_d$).
	\label{h}

The map $(\mathbf{m},w)\mapsto u^\mathbf{m}\cdot w$ with $(u^\mathbf{m}\cdot w)_\mathbf{n}=w_{\mathbf{n}-\mathbf{m}}$ is a $\mathbb{Z}^d$-action by automorphisms of the additive group $W_d$ which extends linearly to an $R_d$-action on $W_d$ given by
	\begin{equation}
	\label{eq:module}
h\cdot w=\sum_{\mathbf{n}\in\mathbb{Z}^d}h_\mathbf{n}u^{\mathbf{n}}\cdot w
	\end{equation}
for every $h\in R_d$ and $w\in W_d$. If $w$ also lies in $R_d$ this definition is consistent with the usual product in $R_d$.

\medskip For the following discussion we assume that $d\ge 2$ and consider the irreducible Laurent polynomial
	\begin{equation}
	\label{eq:fd}
\smash{f^{(d)}=2d-\sum_{i=1}^d(u_i+u_i^{-1}) \in R_d.}
	\end{equation}
The equation
	\begin{equation}
	\label{eq:inverse}
f^{(d)}\cdot w=1
	\end{equation}
with $w\in W_d$ admits a multitude of solutions.\footnote{\label{delta}Under the obvious embedding of $R_d\hookrightarrow \ell ^\infty (\mathbb{Z}^d,\mathbb{Z})$, the constant polynomial $1\in R_d$ corresponds to the element $\delta ^{(\mathbf{0})}\in\ell ^\infty (\mathbb{Z}^d,\mathbb{Z})$ given by
	$$
\smash[t]{\delta ^{(\mathbf{0})}_\mathbf{n}=
	\begin{cases}
1&\textup{if}\enspace \mathbf{n}=\mathbf{0},
	\\
0&\textup{otherwise}.
	\end{cases}}
	$$}
However, there is a distinguished (or \textit{fundamental}) solution $w^{(d)}$ of \eqref{eq:inverse} which has a deep probabilistic meaning: it is a certain multiple of the \textit{lattice Green's function of the symmetric nearest-neighbour random walk on $\mathbb{Z}^d$} (cf. \cite{deBHR}, \cite{FU}, \cite{S}, \cite{U}).

	\begin{defi}
	\label{d:fundamental}
For every $\mathbf{n}=(n_1,\dots ,n_d)\in\mathbb{Z}^d$ and $\mathbf{t}=(t_1,\dots ,t_d)\in\mathbb{T}^d$ we set $\langle \mathbf{n},\mathbf{t}\rangle =\sum_{j=1}^dn_jt_j\in\mathbb{T}$. We denote by
	\begin{equation}
	\label{eq:F}
F^{(d)}(\mathbf{t})=\sum_{\mathbf{n}\in\mathbb{Z}^d} f_\mathbf{n}^{(d)}e^{2\pi i\langle \mathbf{n},\mathbf{t}\rangle }=2d-2\cdot \sum_{j=1}^d \cos(2\pi t_j),\enspace \mathbf{t}=(t_1,\dots ,t_d)\in\mathbb{T}^d,
	\end{equation}
the Fourier transform of $f^{(d)}$.

\medskip (1) For $d=2$,
	$$
\smash{w_\mathbf{n}^{(2)}\defeq\int_{\mathbb{T}^2} \frac{e^{-2\pi i\langle \mathbf{n},\mathbf{t}\rangle} -1}{F^{(2)}(\mathbf{t})}\,d\mathbf{t} \enspace \textup{for every}\enspace \mathbf{n}\in\mathbb{Z}^2.}
	$$

(2) For $d\ge3$,
	$$
\smash[t]{w_\mathbf{n}^{(d)}\defeq\int_{\mathbb{T}^d} \frac{e^{-2\pi i\langle \mathbf{n},\mathbf{t}\rangle} }{F^{(d)}(\mathbf{t})}\,d\mathbf{t}\enspace \textup{for every}\enspace \mathbf{n}\in\mathbb{Z}^d.}
	$$

The difference in these definitions for $d=2$ and $d>2$ is a consequence of the fact that the simple random walk on $\mathbb{Z}^2$ recurrent, while on higher dimensional lattices it is transient.
	\end{defi}

	\begin{theo}[\cite{deBHR,FU,S,U}]
	\label{t:estimates}
We write $\|\cdot \|$ for the Euclidean norm on $\mathbb{Z}^d$.
	\begin{itemize}
	\item[(i)]
For every $d\ge 2$, $w^{(d)}$ satisfies \eqref{eq:inverse}.
	\item[(ii)]
For $d=2$,
	\begin{equation}
	\label{eq:sharpest}
w_\mathbf{n}^{(2)}=
	\begin{cases}
\hphantom{-}0&\textup{if}\enspace \mathbf{n}=\mathbf{0},
	\\
-\frac1{8\pi } \log \|\mathbf{n}\| - \kappa_2- c_2\frac {\frac{1}{\|\mathbf{n}\|^4}(n_1^4+n_2^4)-\frac 34}{\|\mathbf{n}\|^2} + \mathcal{O}(\|\mathbf{n}\|^{-4})&\textup{if}\enspace \mathbf{n}\ne\mathbf{0},
	\end{cases}
	\end{equation}
where $\kappa_2>0$ and $c_2>0$. In particular, $w_\mathbf{0}^{(2)}=0$ and $w_\mathbf{n}^{(2)}<0$ for all $\mathbf{n}\ne\mathbf 0$. Moreover,
	$$
4\cdot w^{(2)}_\mathbf{n} = \sum_{k=1}^\infty \bigl(\mathbf{P}(X_k=\mathbf{n}|X_0=\mathbf 0) -\mathbf{P}(X_k=\mathbf 0|X_0=\mathbf 0)\bigr),
	$$
where $(X_k)$ is the symmetric nearest-neighbour random walk on $\mathbb{Z}^2$.

	\item[(iii)]
For $d\ge 3$,
	\begin{equation}
	\label{eq:sharpestD}
\|\mathbf{n}\|^{d-2} w_\mathbf{n}^{(d)}= \kappa_d +c_d\frac { \frac{1}{\|\mathbf{n}\|^4}\sum_{i=1}^d n_i^4-\frac 3{d+2}}{\|\mathbf{n}\|^2} +\mathcal{O}(\|\mathbf{n}\|^{-4})
	\end{equation}
as $\|\mathbf{n}\|\to \infty$, where $\kappa_d>0$, $c_d>0$. Moreover,
	$$
2d\cdot w^{(d)}_\mathbf{n} = \sum_{k=0}^\infty\mathbf{P}(X_k=\mathbf n|X_0=\mathbf 0)>0\enspace \textup{for every}\enspace \mathbf{n}\in\mathbb{Z}^d,
	$$
where $(X_k)$ is again the symmetric nearest-neighbour random walk on $\mathbb{Z}^d$.
	\end{itemize}
	\end{theo}

	\begin{defi}
	\label{d:ideal}
Let $w^{(d)}\in W_d$ be the point appearing in Definition \ref{d:fundamental}. We set
	\begin{equation}
	\label{eq:Id}
I_d=\bigl\{g\in R_d:g\cdot w^{(d)}\in \ell ^1(\mathbb{Z}^d)\bigr\}\supset (f^{(d)}),
	\end{equation}
where $(f^{(d)})=f^{(d)}\cdot R_d$ is the principal ideal generated by $f^{(d)}$. Since $w^{(d)}_\mathbf{n}=w^{(d)}_{-\mathbf{n}}$ for every $\mathbf{n}\in\mathbb{Z}^d$ it is clear that $I_d=I_d^*=\{g^*:g\in I_d\}$.
	\end{defi}

	\begin{theo}
	\label{t:homoclinic}
The ideal $I_d$ is of the form
	\begin{equation}
	\label{eq:Idform}
I_d=(f^{(d)})+\mathscr{I}_d^3,\vspace{-2mm}
	\end{equation}
where
	\begin{equation}
	\label{eq:maximal}
\mathscr{I}_d=\bigl\{h\in R_d:h(\mathbf{1})=0\bigr\}= (1-u_1)\cdot R_d+\dots +(1-u_d)\cdot R_d
	\end{equation}
with $\mathbf{1}=(1,\dots ,1)$.
	\end{theo}

For the proof of Theorem \ref{t:homoclinic} we need several lemmas. We set
	\begin{equation}
	\label{eq:J}
J_d=(f^{(d)})+\mathscr{I}_d^3\subset R_d.
	\end{equation}

	\begin{lemm}
	\label{l:ideals}
Let $g=\sum_{\mathbf{k}\in\mathbb{Z}^d}g_\mathbf{k} u^\mathbf{k}\in R_d$. Then $g\in J_d$ if and only if it satisfies the following conditions \eqref{eq:condA}--\eqref{eq:condD}.
	\begin{align}
\sum_{\mathbf{k}\in\mathbb{Z}^d} g_\mathbf{k}&=0,
	\label{eq:condA}
	\\
\sum_{\mathbf{k}=(k_1,\dots ,k_d)\in\mathbb{Z}^d}g_\mathbf{k} k_i & =0 \qquad \textup{for} \enspace i=1,\ldots,d,
	\label{eq:condB}
	\\
\sum_{\mathbf{k}=(k_1,\dots ,k_d)\in\mathbb{Z}^d} g_\mathbf{k} k_i k_j&=0 \qquad \textup{for}\enspace 1\le i\ne j\le d,
	\label{eq:condC}
	\\
\sum_{\mathbf{k}=(k_1,\dots ,k_d)\in\mathbb{Z}^d} g_\mathbf{k} (k_i^2-k_j^2)&=0 \qquad \textup{for}\enspace 1\le i\ne j\le d.
	\label{eq:condD}
	\end{align}
	\end{lemm}

	\begin{proof}
Condition \eqref{eq:condA} is equivalent to saying that $g\in\mathscr{I}_d$. In conjunction with \eqref{eq:condA}, \eqref{eq:condB} is equivalent to saying that $g\in\mathscr{I}_d^2$: indeed, if $g\in\mathscr{I}_d$, then it is of the form
	\begin{equation}
	\label{eq:hi}
\smash[b]{g=\sum_{i=1}^d (1-u_i)\cdot a_i}
	\end{equation}
with $a_i\in R_d$ for $i=1,\dots ,d$. Then
	$$
\frac{\partial g}{\partial u_j}=\sum_{\mathbf{k}=(k_1,\dots ,k_d) \in\mathbb{Z}^d}g_\mathbf{k}k_j\cdot u_1^{k_1}\cdots u_j^{k_j-1}\cdots u_d^{k_d}= -a_j+\sum_{i=1}^d(1-u_i)\cdot \frac{\partial a_i}{\partial u_j},
	$$
and $\frac{\partial g}{\partial u_j}(\mathbf{1})=0$ if and only if $a_j\in\mathscr{I}_d$.

If $g\in\mathscr{I}_d$ is of the form \eqref{eq:hi} and satisfies \eqref{eq:condB} we set
	\begin{equation}
	\label{eq:aj}
a_j=\sum_{i=1}^d(1-u_i)\cdot b_{i,j}
	\end{equation}
with $b_{i,j}\in R_d$. Condition \eqref{eq:condC} is satisfied if and only if
	$$
\frac{\partial ^2g}{\partial u_i\partial u_j}(\mathbf{1})=-\frac{\partial a_i}{\partial u_j} - \frac{\partial a_j}{\partial u_i}=b_{i,j}(\mathbf{1})+b_{j,i}(\mathbf{1})=0
	$$
for $1\le i\ne j\le d$.

Finally, if $g$ satisfies \eqref{eq:condA}--\eqref{eq:condB} and is of the form \eqref{eq:hi}--\eqref{eq:aj} with $b_{i,j}\in R_d$ for all $i,j$, then \eqref{eq:condD} is equivalent to the existence of a constant $c\in \mathbb R$ with
	$$
\sum_{\mathbf{k}=(k_1,\dots ,k_d)\in\mathbb{Z}^d}g_\mathbf{k} k_i^2 = -2\frac{\partial a_i}{\partial u_i}(\mathbf{1})=2b_{i,i}(\mathbf{1})=c
	$$
for $i=1,\dots ,d$.

The last equation shows that $b_{i,i}-b_{1,1}\in \mathscr{I}_d$ for $i=2,\dots ,d$. By combining all these observations we have proved that $g$ satisfies \eqref{eq:condA}--\eqref{eq:condD} if and only if it is of the form
	\begin{equation}
	\label{eq:gform}
g=h_1\cdot \sum_{i=1}^d (1-u_i)^2 + h_2
	\end{equation}
with $c\in\mathbb{Z}$, $h_1 \in R_d$ and $h_2\in\mathscr{I}_d^3$. The set of all such $g\in R_d$ is an ideal which we denote by $\tilde{J}$. Clearly, $ \mathscr{I}_d^3\subset \tilde{J}$ and $\sum_{i=1}^d(1-u_i)^2\in \tilde{J}$. Since $(1-u_i)^2\cdot (1-u_i^{-1})\in\mathscr{I}_d^3$ for $i=1,\dots ,d$ as well, we conclude that
	\begin{equation}
	\label{eq:idealequality}
f^{(d)}=\sum_{i=1}^d(1-u_i)^2- \sum_{i=1}^d(1-u_i^{-1})\cdot (1-u_i)^2\in \tilde{J}.
	\end{equation}
This shows that $\tilde{J}\subset J_d$, and the reverse inclusion also follows from \eqref{eq:idealequality} and \eqref{eq:gform}.
	\end{proof}

	\begin{lemm}
	\label{l:necessary}
$I_d\subset J_d$.
	\end{lemm}

	\begin{proof}
We assume that $g\in I_d$ and set $v=g\cdot w^{(d)}$. In order to verify \eqref{eq:condA} we argue by contradiction and assume that $\sum_\mathbf{k} g_\mathbf{k}\ne 0$. If $d=2$ then
	$$
v_\mathbf{n} = -\frac {\sum_\mathbf{k} g_\mathbf{k}}{2\pi}\log \|\mathbf{n}\| + \textup{l.o.t.},
	$$
for large $\|\mathbf{n}\|$. If $d\ge 3$, then
	$$
v_\mathbf{n} = \frac {\kappa_d \sum_\mathbf{k} g_\mathbf{k}} {\|\mathbf{n}\|^{d-2}}+\textup{l.o.t.}
	$$
for large $\|\mathbf{n}\|$. In both cases it is evident that $v\not\in \ell ^1(\mathbb{Z}^d)$.

By taking \eqref{eq:condA} into account one gets that, for every $d\ge 2$,
	\begin{align*}
v_\mathbf{n}=(g\cdot w^{(d)})_\mathbf{n} & =\sum_{\mathbf{k}} g_\mathbf{k} w_{\mathbf{n}-\mathbf{k}}^{(d)}
	\\
&=\int_{\mathbb{T}^d} e^{-2\pi i \langle \mathbf{n},\mathbf{t}\rangle }\frac {\sum_{\mathbf{k}} g_\mathbf{k} e^{2\pi i\langle \mathbf{k},\mathbf{t}\rangle }} {2d-2\sum_{j=1}^d\cos(2\pi t_j)} \,d\mathbf{t}.
	\end{align*}
Hence $v=(v_\mathbf{n})$ is the sequence of Fourier coefficients of the function
	$$
H(\mathbf{t}) = \frac {\sum_{\mathbf{k}} g_\mathbf{k} e^{2\pi i\langle \mathbf{k},\mathbf{t}\rangle }} {2d-2\sum_{j=1}^d\cos(2\pi t_j)}.
	$$

If $v\in \ell ^1(\mathbb{Z}^d)$, then $H$ must be a continuous function on $\mathbb{T}^d$. Since $\mathbf{t}=\mathbf{0}$ is the only zero of $F^{(d)}$ on $\mathbb{T}^d$ (cf. \eqref{eq:F}), the numerator $G=\sum_{k} g_\mathbf{k} e^{2\pi i \langle\mathbf{k},\cdot \rangle }$ must compensate for this singularity. Consider the Taylor series expansion of $G$ at $\mathbf{t}=\mathbf{0}$:
	$$
G(\mathbf{t})=\sum_{k} g_\mathbf{k} +2\pi i \sum_{j=1}^d t_j \sum_{\mathbf{k}} g_\mathbf{k} k_j-2\pi^2 \sum_{j=1}^d t_j^2 \sum_{\mathbf{k} }g_\mathbf{k} k_j^2 -4\pi^2 \sum_{i\ne j} t_i t_j\sum_{\mathbf{k}} g_\mathbf{k} k_ik_j+\textup{h.o.t.}
	$$
The Taylor series expansion of $F^{(d)}$ at $\mathbf{t}=\mathbf{0}$ is given by
	$$
\smash[b]{F^{(d)}(\mathbf{t})= 4\pi^2\sum_{j=1}^d t_j^2+\textup{h.o.t.}}
	$$
Suppose that
	$$
h(\mathbf{t}) =\frac {a_0+\sum_{j=1}^d b_j t_j+ \sum_{j=1}^d c_j t_j^2+\sum_{i\ne j} d_{i,j} t_i t_j+\textup{h.o.t}}{t_1^2+\dots+t_d^2+\textup{h.o.t}}
	$$
is continuous at $\mathbf{t}=\mathbf{0}$. Then
	$$
a_0 =0,\quad b_j=0\enspace \enspace \textup{for all}\;j, \quad c_j=c \enspace \enspace \textup{for all}\;j,\quad d_{ij}=0\quad\textup{for all}\;i\ne j,
	$$
and for some constant $c$. If any of these conditions is violated, then one easily produces examples of sequences $\mathbf{t}^{(m)}\to\mathbf{0}$ as $m\to\infty$ with distinct limits $\lim_{m\to\infty} h(\mathbf{t}^{(m)})$. By applying this to $H$ we obtain \eqref{eq:condA}--\eqref{eq:condD}, so that $g\in J_d$ by Lemma \ref{l:ideals}.
	\end{proof}

To establish the inclusion $J_d\subseteq I_d$, we have to show that for any $g\in J_d$, $g\cdot u\in\ell^1(\mathbb{Z}^d)$ where $u\in W_d$ of the form
	$$
\omega _\n=\frac{\sum_{i=1}^d n_i^4}{\|\n\|^{d+4}},\enspace \enspace \textup{or}\enspace \enspace \omega _\n=\frac 1{\|\n\|^{\gamma}}\enspace \enspace \textup{with} \enspace \gamma\ge d-2.
	$$
For $d=2$, we also have to treat the case $\omega _\n=\log\|\n\|$.

These results are obtained in the following three lemmas.

	\begin{lemm}
	\label{l:easycase}
Suppose that $d\ge 2$ and that $\omega \in W_d$ is given by
	$$
\omega _{\mathbf{n}} =
	\begin{cases}
0&\textup{if}\enspace\mathbf{n}=\mathbf{0},
	\\
\frac {\sum_{i=1}^d n_i^4}{\|\n\|^{d+4}}&\textup{if}\enspace\mathbf{n}\ne \mathbf{0}.
	\end{cases}
	$$
If $g\in R_d$ satisfies \eqref{eq:condA}, then $g\cdot \omega \in \ell^{1}(\mathbb Z^d)$.
	\end{lemm}

	\begin{proof}
Let $M=\max\{\|\k\|: g_\k\ne 0\}$, and suppose that $\|\n\|>M$. Then
	\begin{align*}
(g\cdot \omega )_\n&=\sum_{\k} g_k\frac {\sum_{i=1}^d (n_i-k_i)^4} {\|\n-\k\|^{d+4}}= \sum_{\k} g_k\frac{\sum_{i=1}^d n_i^4 +\mathcal O(\|\n\|^3)}{ \|\n\|^{d+4} (1+\mathcal O(\|\n\|^{-1}))}
	\\
&=\frac{\sum_{i=1}^d n_i^4}{ \|\n\|^{d+4}}\biggl( \sum_\k g_\k\biggr)+\mathcal O\biggl( \frac {1}{\|\n\|^{d+1}}\biggr)=\mathcal O\biggl( \frac {1}{\|\n\|^{d+1}}\biggr).
	\end{align*}
Therefore, $\sum_{\n} |(g\cdot \omega )_\n|<\infty$.
	\end{proof}

For the reverse inclusion $J_d\subset I_d$ we need different arguments for $d=2$ and for $d\ge3$. We start with the case $d=2$.

	\begin{lemm}
	\label{l:embed}
Suppose that $g=\sum_{\mathbf{k}\in\mathbb{Z}^2} g_\mathbf{k}u^\mathbf{k}\in R_2$ satisfies \eqref{eq:condA}. We set $S_+=\{\mathbf{k}: g_\mathbf{k}>0\}$ and $S_-=\{\mathbf{k}: g_\mathbf{k}<0\}$. Put
	$$
M_g= 2\sum_{\mathbf{k}\in S_+} g_\mathbf{k} = 2\sum_{\mathbf{k}\in S_-} |g_\mathbf{k}|
	$$
and define two polynomials in the variables $(n_1,n_2)$:
	\begin{equation}
	\label{eq:P+-}
	\begin{aligned}
P_+(n_1,n_2) & =\prod_{\mathbf{k}\in S_+}\bigl( (n_1-k_1)^2+(n_2-k_2)^2\bigr)^{g_{\mathbf{k}}} = \prod_{\mathbf{k}\in S_+}\|\mathbf{n}-\mathbf{k}\|^{2g_{\mathbf{k}}},
	\\
P_-(n_1,n_2) &= \prod_{\mathbf{k}\in S_-}\bigl( (n_1-k_1)^2+(n_2-k_2)^2\bigr)^{|g_\mathbf{k}|}= \prod_{\mathbf{k}\in S_-}\|\mathbf{n}-\mathbf{k}\|^{2|g_{\mathbf{k}}|}.
	\end{aligned}
	\end{equation}
Let $m_g$ be the degree of $P= P_+-P_-$. If
	\begin{equation}
	\label{eq:degree}
M_g-m_g\ge 3,
	\end{equation}
then $g\cdot \omega \in \ell^1(\mathbb Z^2)$, where
	$$
\omega _{\mathbf{n}} =
	\begin{cases}
0&\textup{if}\enspace\mathbf{n}=(0,0),
	\\
\log\|\n\|&\textup{if}\enspace\mathbf{n}\ne (0,0).
	\end{cases}
	$$
	\end{lemm}

	\begin{proof}
Since $\sum_{\mathbf{k}\in\mathbb{Z}^2} g_\mathbf{k} =0$ by \eqref{eq:condA}, $M_g=\deg P_+=\deg P_-$ and
	$$
m_g=\deg P<\max(\deg P_+,\deg P_-)=M_g.
	$$
Let $v=g\cdot \omega $. Hence, for all $\n$ with $\|\n\|>\max\{\|\k\|: \k\in S_+\cup S_\}$, one has
	$$
|(g\cdot \omega )_\mathbf{n}|= \frac 1{2}\biggl| \log\frac {P_{+}(n_1,n_2)}{ P_{-}(n_1,n_2)}\biggr| =\frac 1{2}\biggl| \log\biggl(1+\frac {P_{+}(n_1,n_2)-P_{-}(n_1,n_2)}{ P_{-}(n_1,n_2)}\biggr)\biggr|.
	$$
There exist constants $C,N$ such that
	$$
\biggl|\frac {P_{+}(n_1,n_2)-P_{-}(n_1,n_2)}{ P_{-}(n_1,n_2)} \biggr|\le C\frac{\|\mathbf{n}\|^{m_g}}{\|\mathbf{n}\|^{M_g}}=\frac {C}{\|\mathbf{n}\|^{M_g-m_g}}<\frac12
	$$
for $\|\mathbf{n}\|\ge N$. Hence we can find another constant $\tilde C$ such that
	$$
|(g\cdot \omega )_\n| \le \frac {\tilde C}{\|\mathbf{n}\|^{M_g-m_g}}
	$$
for all sufficiently large $\|\n\|$. Since $M_g-m_g\ge 3$, we finally conclude that $g\cdot \omega \in \ell^1(\mathbb Z^2)$.
	\end{proof}

	\begin{lemm}
	\label{l:polynomcase}
Suppose that $g\in J_d$ \textup{(}cf. \eqref{eq:condA}--\eqref{eq:condD}\textup{)}, and that $\omega \in W_d$ is given by
	$$
\omega _{\mathbf{n}} =
	\begin{cases}
0&\textup{if}\enspace\mathbf{n}=\mathbf{0},
	\\
\frac {1}{\|\n\|^{\gamma}} &\textup{if}\enspace\mathbf{n}\ne \mathbf{0},
	\end{cases}
	$$
for some integer $\gamma\ge d-2$. Then $g\cdot \omega \in \ell^1(\mathbb Z^d)$.
	\end{lemm}

	\begin{proof}
Let $S_g=\{\mathbf{k}\in \mathbb{Z}^d: g_\mathbf{k}\ne 0\}$, $M=\max\{\|\k\|: \k\in S_g\}$, and note that
	\begin{equation}
	\label{eq:B2}
S_g\subset B_d= \{\mathbf{y}\in\mathbb{R}^d: \|\mathbf{y}\| \le M\},
	\end{equation}
where $\|\cdot \|$ is the Euclidean norm on $\mathbb{Z}^d\subset \mathbb{R}^d$.

We fix $\mathbf{n}\in\mathbb{Z}^d$ with $\|\mathbf{n}\|>M$ and set
	\begin{equation}
	\label{eq:h(n)}
h^{(\mathbf{n})}(\mathbf{k})= \|\mathbf{n}-\mathbf{k}\|^{-\gamma }=\biggl(\sum_{i=1}^d (n_i-k_i)^2\biggr)^ {-\gamma /2}.
	\end{equation}
In calculating the Taylor expansion of $h^{(\mathbf{n})}$ as a function of the variables $k_1,\dots ,k_d$ we use the notation
	\begin{equation}
	\label{eq:notation}
I!=i_1!\cdots i_d!,\enspace \enspace |I|=i_1+\dots +i_d\enspace \enspace \textup{and}\enspace \enspace \frac{\partial ^{|I|} h^{(\mathbf{n})}}{\partial \mathbf{k}^I}=\frac{\partial^{i_1+\dotsb+i_d} h^{(\mathbf{n})}} {\partial k_1^{i_1} \cdots \partial k_n^{i_n}},
	\end{equation}
for $I=(i_1,\dots ,i_d)\in\mathbb{Z}_+^d, \;\mathbf{k}=(k_1,\dots ,k_d)\in\mathbb{Z}^d$, where $\mathbb{Z}_+=\{n\in\mathbb{Z}:n\ge0\}$. Then the Taylor expansion of $h^{(\mathbf{n})}$ for $\|\mathbf{k}\|\le M$ is given by
	$$
\smash[b]{h^{(\mathbf{n})}(\mathbf{k}) = \sum_{|I| \leq 2} \frac{1}{I!} \frac{\partial^{|I|} h^{(\mathbf{n})}}{\partial \mathbf{k}^I}(\mathbf{0}) \, \mathbf{k}^I+\sum_{|I|=3} R_I^{(\mathbf{n})} \mathbf{k}^I,}
	$$
where
	$$
\smash[t]{|R_I^{(\mathbf{n})}|\le\sup_{\mathbf{y}\in B_d}\biggl|\frac{1}{I!} \frac{\partial^{|I|}h^{(\mathbf{n})}}{\partial\mathbf{k}^I}(\mathbf{y})\biggr|}
	$$
(cf. \eqref{eq:B2}).

The first and second order derivatives of $h^{(\mathbf{n})}$ have the following form.
	\begin{align*}
\frac{\partial h^{(\mathbf{n})}}{\partial k_i}(\mathbf{k})&=\gamma \cdot (n_i-k_i) \cdot \|\mathbf{n}-\mathbf{k}\|^{-\gamma -2}\enspace \enspace \textup{for}\enspace i=1,\dots ,d,
	\\
\frac{\partial^2 h^{(\mathbf{n})}}{\partial k_i\partial k_j}(\mathbf{k})&=\gamma \cdot (\gamma +2)\cdot (n_i-k_i)\cdot (n_j-k_j)\cdot \|\mathbf{n}-\mathbf{k}\|^{-\gamma -4} \enspace \enspace \textup{for}\enspace i,j=1,\dots ,d,\;i\ne j,
	\\
\frac{\partial^2 h^{(\mathbf{n})}}{\partial k_i^2}(\mathbf{k})&=\gamma \cdot (\gamma +2)\cdot (n_i-k_i)^2 \cdot \|\mathbf{n}-\mathbf{k}\|^{-\gamma -4} - \gamma \cdot \|\mathbf{n}-\mathbf{k}\|^{-\gamma -2}\enspace \enspace \textup{for}\enspace i=1,\dots ,d.
	\end{align*}
It follows that
	\begin{align*}
h^{(\mathbf{n})}(\mathbf{0}) &=\|\mathbf{n}\|^{-\gamma },
	\\
\frac{\partial h^{(\mathbf{n})}}{\partial k_i}(\mathbf{0})&=\gamma \cdot n_i\cdot \|\mathbf{n}\| ^{-\gamma -2},
	\\
\frac{\partial^2 h^{(\mathbf{n})}}{\partial k_i\partial k_j}(\mathbf{0})&= \gamma \cdot (\gamma +2)\cdot n_i\cdot n_j\cdot \|\mathbf{n}\|^{-\gamma -4}, \quad i\ne j,
	\\
\frac{\partial^2 h^{(\mathbf{n})}}{\partial k_i^2}(\mathbf{0}) &=\gamma \cdot (\gamma +2) \cdot n_i^2 \cdot \|\mathbf{n}\|^{-\gamma -4} - \gamma \cdot \|\mathbf{n}\|^{-\gamma -2}.
	\end{align*}
For $I=(i_1,\dots ,i_d)\in\mathbb{Z}_+^d$ and $\mathbf{y}\in\mathbb{R}^d$,
	$$
\frac{\partial^{|I|}h^{(\mathbf{n})}}{\partial \mathbf{k}^I}(\mathbf{y}) =P_I(n_1,\ldots,n_d)\cdot \|\mathbf{n}-\mathbf{y}\|^{-\gamma -2|I|},
	$$
where $P_I$ is a polynomial of degree at most $|I|$ in the variables $n_1,\dots ,n_d$. Therefore, for every $I\in\mathbb{Z}_+^d$ with $|I|=3$,
	\begin{equation}
	\label{eq:C}
|R_I^{(\mathbf{n})}|\le \mathcal O( \|\mathbf{n}\|^{-\gamma -3}).
	\end{equation}

\medskip By using the Taylor series expansion of $h^{(\mathbf{n})}$ above we obtain that, for all $\n$ with sufficiently large norm,
	\begin{align}
|(g\cdot \omega )_\mathbf{n}|&= \biggl|\sum_{\mathbf{k}\in S_g} g_\mathbf{k} h^{(\n)}(\k) \biggr|\notag
	\\
&\le \biggl| h^{(n)}(\mathbf{0})\sum_{\k\in S_g} g_\k \biggr| + \biggl| \sum_{i=1}^d \frac {\partial h^{(\n)}(\mathbf{0})} {\partial k_i}\biggl(\sum_{\k\in S_g} g_\k k_i\biggr)\biggr|+ \biggl| \sum_{i\ne j} \frac {\partial^2 h^{(\n)}(\mathbf{0})} {\partial k_i\partial k_j}\biggl(\sum_{\k\in S_g} g_\k k_ik_j\biggr)\biggr|\notag
	\\
&\quad+\frac 12\biggl| \sum_{i=1}^d \frac {\partial^2 h^{(\n)}(\mathbf{0})} {\partial k_i^2}\biggl(\sum_{\k\in S_g} g_\k k_i^2\biggr)\biggr| +\mathcal O(\|\n\|^{-(\gamma+3)}).
	\label{eq:longeq}
	\end{align}
The first three terms on the right hand side of the above inequality vanish because of \eqref{eq:condA}, \eqref{eq:condB}, and \eqref{eq:condC}. The fourth term is estimated as follows: \eqref{eq:condD} implies that
	$$
\sum_{\k\in S_g} g_\k k_i^2 = \textup{const}\quad\textup{for all}\enspace i=1,\ldots,d,
	$$
and we denote by $C$ this common value. Then
	\begin{align*}
\sum_{i=1}^d &\frac {\partial^2 h^{(\n)}(\mathbf{0})} {\partial k_i^2}\biggl(\sum_{\k\in S_g} g_\k k_i^2\biggr)
	\\
&= \sum_{i=1}^d\bigl(\gamma(\gamma+2)\cdot n_i^2\cdot \|\mathbf{n}\|^{-\gamma-4}-\gamma\cdot \|\mathbf{n}\|^{-\gamma-2}\bigr)C=\bigl[ \gamma(\gamma+2)- \gamma d\bigr]C ||\n||^{-\gamma-2}.
	\end{align*}
Therefore, if $\gamma=d-2$, then the fourth term vanishes. If $\gamma>d-2$, i.e., if $\gamma\ge d-1$, then the fourth term is of the order $\mathcal O(\|\n\|^{-(d+1)})$, and is thus summable over $\mathbb Z^d$. The remainder term in \eqref{eq:longeq} is always summable since $\gamma+3\ge d+1$.
	\end{proof}

	\begin{proof}[Proof of Theorem \ref{t:homoclinic}]
We start with the case $d\ge 3$. Recall that for $\n\ne \mathbf{0}$
	$$
w_{\n}^{(d)} = \frac{\kappa_d}{\|\mathbf{n}\|^{d-2}} +c_d \frac {\sum_{i=1}^d n_i^4}{\|\n\|^{d+4}}-\frac {3c_d}{d+2}\frac {1}{\|\mathbf{n}\|^{d}} +\mathcal{O}(\|\mathbf{n}\|^{-(d+2)}) =: \omega ^{(1)}_\n+\omega ^{(2)}_\n+\omega ^{(3)}_\n+r_\n.
	$$
Applying $g$, we conclude that $g\cdot w\in\ell^1(\mathbb Z^d)$, because $g\cdot \omega ^{(1)}, g\cdot \omega ^{(3)}\in \ell^1(\mathbb Z^d)$ by Lemma \ref{l:polynomcase} for $\gamma=d-2$ and $\gamma=d$, respectively; $g\cdot \omega ^{(2)}\in \ell^1(\mathbb Z^d)$ by Lemma \ref{l:easycase}; $(g\cdot r)_\n=\mathcal O(\|\n\|^{-(d+2)})$, and hence $g\cdot r\in \ell^1(\mathbb Z^d)$ as well.

Now consider the case $d=2$. Then
	$$
w_\n^{(2)}=-\frac1{8\pi } \log \|\mathbf{n}\| - \kappa_2- c_2\frac {n_1^4+n_2^4}{\|\n\|^{4+2}}-\frac 34\frac {1}{\|\mathbf{n}\|^2} + \mathcal{O}(\|\mathbf{n}\|^{-4})=\omega ^{(1)}_\n+\omega ^{(2)}_\n+\omega ^{(3)}_\n+r_\n.
	$$
For any $g\in J_2$,
	\begin{equation}
	\label{eq:encl2}
g\cdot \omega ^{(2)},\;g\cdot \omega ^{(3)},\;g\cdot r\in \ell^1(\mathbb Z^2)
	\end{equation}
by the results of the Lemmas \ref{l:easycase} and \ref{l:polynomcase}.

The remaining term $g\cdot \omega ^{(1)}$ has to be treated slightly differently. First of all, note that since
	$$
J_2=(f)+(u_1-1)^3\cdot R_2+(u_1-1)^2(u_2-1)\cdot R_2+(u_1-1)(u_2-1)^2\cdot R_2 +(u_2-1)^3\cdot R_2,
	$$
it is sufficient to check that $g\cdot \omega ^{(1)}\in\ell^1(\mathbb Z^2)$ only for the set of generators, i.e., for
	$$
g=f^{(2)},\; (u_1-1)^3,\;(u_1-1)^2(u_2-1),\;(u_1-1)(u_2-1)^2,\;(u_2-1)^3.
	$$

For $g=f^{(2)}$, $f^{(2)}\cdot w^{(2)} =\delta^{(\mathbf{0}})\in\ell^1(\mathbb Z^2)$ (cf. \eqref{eq:inverse} and Footnote \ref{delta} \vpageref{delta}), and hence, given \eqref{eq:encl2}, $f\cdot \omega ^{(1)}\in \ell^{1}(\mathbb Z^2)$ as well.

For $g=(u_1-1)^3\in R_2$ we apply Lemma \ref{l:embed}. Note that $S_+=\{(1,0),(3,0)\}$, $S_{-}=\{(0,0),(2,0)\}$,
	$$
P_+=((n_1 - 3)^2 + n_2^2)((n_1 - 1)^2 + n_2^2)^3,\enspace P_-=((n_1 - 2)^2 + n_2^2)^3(n_1^2 + n_2^2)
	$$
and
	\begin{align*}
P_+-P_i=9 - 60n_1 &+ 108n_1^2 - 84n_1^3 + 30n_1^4 - 4n_1^5 - 36n_2^2 + 60n_1n_2^2
	\\
&- 36n_1^2n_2^2 + 8n_1^3n_2^2 - 18n_2^4 + 12n_1n_2^4
	\end{align*}
Hence $M_g=\deg P_+=\deg P_-=8$, $m_g=\deg P=5$, $M_g-m_g=3$. Therefore, by Lemma \ref{l:embed}, $|(g\cdot \omega ^{(1)})_\n|=\mathcal O(\|\n\|^{-3})$, and hence $g\cdot \omega ^{(1)}\in\ell^1(\mathbb Z^2)$, which is equivalent to $g\in I_2$.

The same calculation shows that $(u_2-1)^3\in I_2$. Furthermore, since $f^{(2)}\in I_2$ and
	$$
u_1^{-1}(u_1-1)^3+f^{(2)}=-u_2^{-1}(u_1-1)(u_2-1)^2,
	$$
we obtain that $(u_1-1)(u_2-1)^2\in I_2$ and, by symmetry, that $(u_1-1)^2(u_2-1)\in I_2$. This proves that $J_2\subset I_2$, and Lemma \ref{l:necessary} yields that $J_2=I_2$.
	\end{proof}

\section{The harmonic model
	\label{s:harmonic}
}

Let $d>1$. We define the \textit{shift-action} $\alpha $ of $\mathbb{Z}^d$ on $\mathbb{T}^{\mathbb{Z}^d}$ by
	\begin{equation}
	\label{eq:alpha}
(\alpha ^\mathbf{m}x)_\mathbf{n}=x_{\mathbf{m}+\mathbf{n}}
	\end{equation}
for every $\mathbf{m},\mathbf{n}\in\mathbb{Z}^d$ and $x=(x_\mathbf{n})\in \mathbb{T}^{\mathbb{Z}^d}$ and consider, for every $h\in R_d$, the group homomorphism
	\begin{equation}
	\label{eq:halpha}
h(\alpha )=\sum_{\mathbf{m}\in\mathbb{Z}^d}h_\mathbf{m}\alpha ^\mathbf{m} \colon \mathbb{T}^{\mathbb{Z}^d}\longrightarrow \mathbb{T}^{\mathbb{Z}^d}.
	\end{equation}
Since $R_d$ is an integral domain, Pontryagin duality implies that $h(\alpha )$ is surjective for every nonzero $h\in R_d$ (it is dual to the injective homomorphism from $R_d\cong \widehat{\mathbb{T}^{\mathbb{Z}^d}}$ to itself consisting of multiplication by $h$).

Let $f^{(d)}\in R_d$ be given by \eqref{eq:fd} and let $X_{f^{(d)}}\subset \mathbb{T}^{\mathbb{Z}^d}$ be the closed, connected, shift-invariant subgroup
	\begin{equation}
	\label{eq:harmonic}
	\begin{aligned}
X_{f^{(d)}}=\ker f^{(d)}(\alpha )=\biggl\{x=(x_\mathbf{n})\in\mathbb{T}^{\mathbb{Z}^d}&:2dx_{\mathbf{n}}-\smash{\sum_{j=1}^d(x_{\mathbf{n}+\mathbf{e}^{(j)}}+x_{\mathbf{n}-\mathbf{e}^{(j)}})=0}
	\\
&\qquad \qquad \qquad \qquad \enspace \enspace \smash[t]{\textup{for every}\enspace \mathbf{n}\in\mathbb{Z}^d\biggr\}}.
	\end{aligned}
	\end{equation}

\vspace{-2mm}\noindent We denote by $\alpha _{f^{(d)}}$ the restriction of $\alpha $ to $X_{f^{(d)}}$. Since every $\alpha _{f^{(d)}}^\mathbf{m},\,\mathbf{m}\in\mathbb{Z}^d$, is a continuous automorphism of $X_{f^{(d)}}$, the $\mathbb{Z}^d$-action $\alpha _{f^{(d)}}$ preserves the normalized Haar measure $\lambda _{X_{f^{(d)}}}$ of $X_{f^{(d)}}$.

The Laurent polynomial $f^{(d)}$ can be viewed as a Laplacian on $\mathbb{Z}^d$ and every $x=(x_\mathbf{n})\in X_{f^{(d)}}$ is \textit{harmonic} (mod 1) in the sense that, for every $\mathbf{n}\in\mathbb{Z}^d$, $2d\cdot x_\mathbf{n}$ is the sum of its $2d$ neighbouring coordinates (mod 1). This is the reason for calling $(X_{f^{(d)}},\alpha _{f^{(d)}})$ the $d$-dimensional \textit{harmonic model}.

According to \cite[Theorem 18.1]{DSAO} and \cite[Theorem 19.5]{DSAO}, the metric entropy of $\alpha _{f^{(d)}}$ with respect to $\lambda _{X_{f^{(d)}}}$ coincides with the topological entropy of $\alpha _{f^{(d)}}$ and is given by
	\begin{equation}
	\label{eq:entropy}
h_{\lambda _{X_{f^{(d)}}}}(\alpha _{f^(d)})=h_\textup{top}(\alpha _{f^{(d)}})=\int_0^1\cdots\int_0^1 \log\,f^{(d)}(2\pi it_1,\dots ,2\pi it_d)\,\,dt_1\cdots dt_d<\infty .
	\end{equation}
Furthermore, $\alpha _{f^{(d)}}$ is Bernoulli with respect to $\lambda _{X_{f^{(d)}}}$ (cf. \cite{DSAO}).

Since every constant element of $\mathbb{T}^{\mathbb{Z}^d}$ lies in $X_{f^{(d)}}$, $\alpha _{f^{(d)}}$ has uncountably many fixed points and is therefore nonexpansive: for every $\varepsilon >0$ there exists a nonzero point $x=(x_\mathbf{n})$ in $X_{f^{(d)}}$ with
	$$
\pmb{|}x_\mathbf{n}\pmb{|}<\varepsilon \enspace \textup{for every}\enspace \mathbf{n}\in\mathbb{Z}^d,
	$$
where
	\begin{equation}
	\label{eq:metric}
\pmb{|}t\,(\textup{mod}\;1)\pmb{|}=\min\,\{|t-n|:n\in\mathbb{Z}\},\enspace t\in\mathbb{R}.
	\end{equation}

\subsection{Linearization}

Consider the surjective map $\rho \colon W_d=\mathbb{R}^{\mathbb{Z}^d}\longrightarrow \mathbb{T}^{\mathbb{Z}^d}$ given by
	\begin{equation}
	\label{eq:rho}
\rho (w)_\mathbf{n}=w_\mathbf{n}\enspace (\textup{mod}\;1)
	\end{equation}
for every $\mathbf{n}\in\mathbb{Z}^d$ and $w=(w_\mathbf{n})\in W_d$. We write $\sigma $ for the shift action
	\begin{equation}
	\label{eq:sigma}
(\sigma ^\mathbf{m}w)_\mathbf{n}=(u^{-\mathbf{m}}\cdot w)_\mathbf{n}=w_{\mathbf{m}+\mathbf{n}}
	\end{equation}
of $\mathbb{Z}^d$ on $W_d$ (cf. \eqref{eq:module}). As in \eqref{eq:halpha} we set, for every $g=\sum_{\mathbf{n}\in\mathbb{Z}^d} g_\mathbf{n} u^\mathbf{n}\in R_d$, $h=\sum_{\mathbf{n}\in\mathbb{Z}^d} h_\mathbf{n} u^\mathbf{n}\in \ell ^1(\mathbb{Z}^d)$,
	\begin{equation}
	\label{eq:hsigma}
h(\sigma )=\sum_{\mathbf{n}\in\mathbb{Z}^d}h_\mathbf{n}\sigma ^\mathbf{n}\colon W_d\longrightarrow W_d.
	\end{equation}
Then
	\begin{equation}
	\label{eq:hsigma2}
	\begin{gathered}
h(\sigma )(w)=h^*\cdot w,
	\\
g(\alpha )(\rho (w))=\rho (g^*\cdot w)
	\end{gathered}
	\end{equation}
for every $w\in W_d$ (cf. \eqref{eq:involution} and \eqref{eq:module}).

\medskip We set $W_d(\mathbb{Z})=\mathbb{Z}^{\mathbb{Z}^d} \subset W_d$. According to \eqref{eq:harmonic},
	\begin{equation}
	\label{eq:Wf}
	\begin{aligned}
W_{f^{(d)}}\,&\hspace{-1mm}\defeq\rho ^{-1}(X_{f^{(d)}})=\{w\in W_d:\rho (w)\in X_{f^{(d)}}\}
	\\
&=f^{(d)}(\sigma )^{-1}(W_d(\mathbb{Z})) =\{w\in W_d:f^{(d)}\cdot w\in W_d(\mathbb{Z})\}.
	\end{aligned}
	\end{equation}
For later use we denote by
	\begin{equation}
	\label{eq:widetilde}
\widetilde{\mathbb{R}}\subset W_d,\enspace \enspace \widetilde{\mathbb{Z}}\subset W_d(\mathbb{Z}),\enspace \enspace \widetilde{\mathbb{T}}\subset \mathbb{T}^{\mathbb{Z}^d}
	\end{equation}
the set of constant elements. If $c$ is an element of $\mathbb{R}$, $\mathbb{Z}$ or $\mathbb{T}$ we denote by $\tilde{c}$ the corresponding constant element of $\widetilde{R}$, $\widetilde{\mathbb{Z}}$ or $\widetilde{\mathbb{T}}$.

Equation \eqref{eq:Wf} allows us to view $W_{f^{(d)}}$ as the \textit{linearization} of $X_{f^{(d)}}$.

\subsection{Homoclinic points} Let $\beta $ be an algebraic $\mathbb{Z}^d$-action on a compact abelian group $Y$, i.e., a $\mathbb{Z}^d$-action by continuous group automorphisms of $Y$. An element $y\in Y$ is \textit{homoclinic} for $\beta $ (or $\beta $-homoclinic to $0$) if $\lim_{\mathbf{n}\to\infty }\beta ^\mathbf{n}y=0$. The set of all homoclinic points of $\beta $ is a subgroup of $Y$, denoted by $\Delta _\beta (Y)$.

If $\beta $ is an expansive algebraic $\mathbb{Z}^d$-action on a compact abelian group $Y$ then $\Delta _\beta (Y)$ is countable, and $\Delta _\beta (Y)\ne\{0\}$ if and only if $\beta $ has positive entropy with respect to the Haar measure $\lambda _Y$ (or, equivalently, positive topological entropy). Furthermore, $\Delta _\beta (Y)$ is dense in $Y$ if and only if $\beta $ has completely positive entropy w.r.t. $\lambda _Y$. Finally, if $\beta $ is expansive, then $\beta ^\mathbf{n}x\to0$ exponentially fast (in an appropriate metric) as $\|\mathbf{n}\|\to\infty $. All these results can be found in \cite{LS}.

If $\beta $ is nonexpansive on $Y$, then there is no guarantee that $\Delta _\beta (Y)\ne\{0\}$ even if $\beta $ has completely positive entropy. Furthermore, $\beta $-homoclinic points $y$ may have the property that $\beta ^\mathbf{n}y\to0$ very slowly as $\|\mathbf{n}\|\to\infty $.

The $\mathbb{Z}^d$-action $\alpha _{f^{(d)}}$ on $X_{f^{(d)}}$ is nonexpansive and the investigation of its homoclinic points therefore requires a little more care. In particular we shall have to restrict our attention to $\alpha _{f^{(d)}}$-homoclinic points $x$ for which $\alpha _{f^{(d)}}^\mathbf{n}x\to0$ sufficiently fast as $\|\mathbf{n}\|\to\infty $. For this reason we set
	\begin{equation}
	\label{eq:homoclinic}
\smash{\Delta _\alpha ^{(1)}(X_{f^{(d)}})=\biggl\{x\in\Delta _\alpha (X_{f^{(d)}}): \sum_{\mathbf{n}\in\mathbb{Z}^d}\pmb{|}x_\mathbf{n}\pmb{|}<\infty \biggr\},}
	\end{equation}
where $\pmb{|}\cdot \pmb{|}$ is defined in \eqref{eq:metric}.

\medskip In order to describe the homoclinic groups $\Delta_\alpha (X_{f^{(d)}})$ and $\Delta _\alpha ^{(1)}(X_{f^{(d)}})$ we set
	\begin{equation}
	\label{eq:xDelta}
x^\Delta =\rho (w^{(d)})\in X_{f^{(d)}}.
	\end{equation}
The fact that $x^\Delta \in X_{f^{(d)}}$ is a consequence of Theorem \ref{t:estimates} (1) and \eqref{eq:Wf}.

	\begin{prop}
	\label{p:Delta}
Let $\alpha _{f^{(d)}}$ be the algebraic $\mathbb{Z}^d$-action on the compact abelian group $X_{f^{(d)}}$ defined in \eqref{eq:harmonic}. Then every homoclinic point $z\in\Delta _\alpha (X_{f^{(d)}})$ is of the form $z=\rho (h\cdot w^{(d)})$ for some $h\in R_d$. Furthermore,
	\begin{equation}
	\label{eq:Delta}
\Delta _\alpha ^{(1)}(X_{f^{(d)}})=\rho \bigl(\{h\cdot w^{(d)}:h\in I_d\}\bigr)
	\end{equation}
\textup{(}cf. Theorem \ref{t:estimates}, \eqref{eq:Id} and \eqref{eq:homoclinic}\textup{)}.
	\end{prop}

	\begin{proof}
If $z\in\Delta_\alpha (X_{f^{(d)}})$, then we choose $w\in\ell ^\infty (\mathbb{Z}^d)$ with $\lim_{\mathbf{n}\to\infty }w_\mathbf{n}=0$ and $\rho (w)=z$. From \eqref{eq:Wf} we know that $f^{(d)}\cdot w\in W_d(\mathbb{Z})$, and the smallness of (most of) the coordinates of $w$ guarantees that $h=f^{(d)}\cdot w\in R_d=\ell ^1(\mathbb{Z}^d)\cap \ell ^\infty (\mathbb{Z}^d,\mathbb{Z})$, where
	$$
\ell ^\infty (\mathbb{Z}^d,\mathbb{Z})=\{w=(w_\mathbf{n})\in\ell ^\infty (\mathbb{Z}^d):w_\mathbf{n}\in\mathbb{Z}\enspace \textup{for every}\enspace \mathbf{n}\in\mathbb{Z}^d\}.
	$$
If we multiply the last identity by $w^{(d)}$ we get that
	$$
w^{(d)}\cdot f^{(d)}\cdot w= w=w^{(d)}\cdot h =h\cdot w^{(d)}
	$$
for some $h\in R_d$.

If $z\in\Delta _\alpha ^{(1)}(X_{f^{(d)}})$ then $w\in\ell ^1(\mathbb{Z}^d)$ and hence, by definition, $h\in I_d$. Conversely, if $h\in I_d$, then $z=\rho (h\cdot w^{(d)})\in\Delta _\alpha ^{(1)}(X_{f^{(d)}})$.
	\end{proof}

	\begin{rema}
	\label{r:fundamental}
A homoclinic point $z$ of an algebraic $\mathbb{Z}^d$-action $\beta $ on a compact abelian group $Y$ is \textit{fundamental} if its homoclinic group $\Delta _\beta (Y)$ is the countable group generated by the orbit $\{\beta ^\mathbf{n}z:\mathbf{n}\in\mathbb{Z}^d\}$ (cf. \cite{LS}).

Proposition \ref{p:Delta} shows that $x^\Delta =\rho (w^{(d)})$ also has the property that its orbit under $\alpha _{f^{(d)}}$ generates the homoclinic groups $\Delta _\alpha (X_{f^{(d)}})$ and $\Delta _\alpha ^{(1)}(X_{f^{(d)}})$, although $x^\Delta $ itself may not be homoclinic (e.g., when $d=2$).
	\end{rema}

\subsection{Symbolic covers of the harmonic model} We construct, for every homoclinic point $z\in \Delta^{(1)}_\alpha (X_{f^{(d)}})$, a shift-equivariant group homomorphism from $\ell ^\infty (\mathbb{Z}^d,\mathbb{Z})$ to $X_{f^{(d)}}$ which we subsequently use to find symbolic covers of $\alpha _{f^{(d)}}$.

According to Proposition \ref{p:Delta}, every homoclinic point $z\in\Delta^{(1)}_\alpha (X_{f^{_{d}}})$ is of the form $z=g(\alpha )(x^\Delta )=\rho (g^*\cdot w^{(d)})$ for some $g\in I_d$. We define group homomorphisms $\bar{\xi }_g\colon \ell ^\infty (\mathbb{Z}^d)\longrightarrow \ell ^\infty (\mathbb{Z}^d)$ and $\xi _g\colon \ell ^\infty (\mathbb{Z}^d)\longrightarrow \mathbb{T}^{\mathbb{Z}^d}$ by
	\begin{equation}
	\label{eq:barxi}
\bar{\xi }_g(w)=(g\cdot w^{(d)})(\sigma )(w)=(g^*\cdot w^{(d)})\cdot w\enspace \enspace \textup{and}\enspace \enspace \xi _g(w)=(\rho \circ \bar{\xi }_g)(w).
	\end{equation}
These maps are well-defined, since
	$$
\bar{\xi }_g(w)_\mathbf{n}=\sum_{\mathbf{k}\in\mathbb{Z}^d} w_{\mathbf{n}-\mathbf{k}}\cdot (g^*\cdot w^{(d)})_\mathbf{k}
	$$
converges for every $\mathbf{n}$, and \textit{equivariant} in the sense that
	\begin{equation}
	\label{eq:equivariance}
	\begin{gathered}
\bar{\xi }_g\circ \sigma ^\mathbf{n}=\sigma ^\mathbf{n}\circ \bar{\xi }_g,\enspace \enspace \xi _g\circ \sigma ^\mathbf{n}=\alpha ^\mathbf{n}\circ \xi _g,
	\\
\bar{\xi }_g\circ h(\sigma )=h(\sigma )\circ \bar{\xi }_g,\enspace \enspace \xi _g\circ h(\sigma )=h(\alpha )\circ \xi _g,
	\end{gathered}
	\end{equation}
for every $\mathbf{n}\in\mathbb{Z}^d$, $g\in I_d$ and $h\in R_d$. We also note that
	$$
\smash[b]{\xi _g(v)=\sum_{\mathbf{n}\in\mathbb{Z}^d} v_\mathbf{n} \alpha ^{-\mathbf{n}}\bigl( g(\alpha)(x^\Delta)\bigr) }
	$$
for every $v=(v_\mathbf{n})\in\ell ^\infty (\mathbb{Z}^d,\mathbb{Z})$.

	\begin{prop}
	\label{p:xig}
For every $g\in I_d$,
	\begin{equation}
	\label{eq:range}
\xi _g(\ell ^\infty (\mathbb{Z}^d,\mathbb{Z}))=
	\begin{cases}
\{0\}&\textup{if}\enspace g\in (f^{(d)}),
	\\
X_{f^{(d)}}&\textup{if}\enspace g\in \tilde{I}_d\defeq I_d\smallsetminus (f^{(d)}),
	\end{cases}
	\end{equation}
\textup{(}cf. \eqref{eq:Id} and \eqref{eq:barxi}--\eqref{eq:equivariance}\textup{)}.
	\end{prop}

We begin the proof of Proposition \ref{p:xig} with two lemmas.

	\begin{lemm}
	\label{l:Wf}
For every $w\in \ell ^\infty (\mathbb{Z}^d)$ and $g\in I_d$,
	\begin{equation}
	\label{eq:lemmaWf}
(f^{(d)}(\sigma )\circ \bar{\xi }_g)(w)=f^{(d)}\cdot (g^*\cdot w^{(d)})\cdot w= g^*\cdot (f^{(d)}\cdot w^{(d)})\cdot w=g^*\cdot w=g(\sigma )(w).
	\end{equation}
Furthermore, $\xi _g(\ell ^\infty (\mathbb{Z}^d,\mathbb{Z}))\subset X_{f^{(d)}}$.
	\end{lemm}

	\begin{proof}
For every $h,v\in R_d$, Theorem \ref{t:estimates} (1) implies that
	\begin{equation}
	\label{eq:timesg}
f^{(d)}\cdot h^*\cdot w^{(d)}\cdot v=h^*\cdot f^{(d)}\cdot w^{(d)}\cdot v=h^*\cdot v.
	\end{equation}

Fix $g\in I_d$ and let $K\ge1$ and $V_K=\{-K+1,\dots ,K-1\}^{\mathbb{Z}^d}\subset \ell ^\infty (\mathbb{Z}^d,\mathbb{Z})$. Then $V_K$ is shift-invariant and compact in the topology of pointwise convergence, and the set $V_K'\subset V_K$ of points with only finitely many nonzero coordinates is dense in $V_K$. For $v\in V_K'\subset R_d$,
	\begin{equation}
	\label{eq:timesg1}
\smash{\bar{\xi }_g(v)=(g^*\cdot w^{(d)})\cdot v}
	\end{equation}
and
	\begin{equation}
	\label{eq:timesg2}
\smash{(f^{(d)}(\sigma )\circ \bar{\xi }_g)(v)=f^{(d)}\cdot g^*\cdot w^{(d)}\cdot v=g^*\cdot f^{(d)}\cdot w^{(d)}\cdot v=g^*\cdot v}
	\end{equation}
by \eqref{eq:barxi} and \eqref{eq:timesg}. Since both $\bar{\xi }_g$ and multiplication by $g^*$ are continuous on $V_K$, \eqref{eq:timesg2} holds for every $v\in V_K$. By letting $K\to\infty $ we obtain \eqref{eq:timesg2} for every $v\in \ell ^\infty (\mathbb{Z}^d,\mathbb{Z})$, hence for every $v\in \frac1{M}\ell ^\infty (\mathbb{Z}^d,\mathbb{Z})$ with $M\ge1$, and finally, again by coordinatewise convergence, for every $w\in \ell ^\infty (\mathbb{Z}^d)$, as claimed in \eqref{eq:lemmaWf}.

For the last assertion of the lemma we note that
	\begin{equation}
	\label{eq:inXf}
\xi _g(v)=\rho ((g^*\cdot w^{(d)})\cdot v)=(g\cdot v^*)(\alpha )(x^\Delta )\in X_{f^{(d)}}
	\end{equation}
for every $v\in V_K'$ (cf.\eqref{eq:xDelta}). The continuity argument above yields that $\xi _g(v)\in X_{f^{(d)}}$ for every $v\in \ell ^\infty (\mathbb{Z}^d,\mathbb{Z})$.
	\end{proof}

	\begin{lemm}
	\label{l:surjective}
If $g\in \tilde{I}_d$ then $\xi _g(\ell ^\infty (\mathbb{Z}^d,\mathbb{Z}))= X_{f^{(d)}}$. In fact,
	$$
\xi _g(\Lambda _{2d})=X_{f^{(d)}},
	$$
where $\Lambda _m=\{0,\dots ,m-1\}^{\mathbb{Z}^d}\subset \ell ^\infty (\mathbb{Z}^d,\mathbb{Z})$ for every $m\ge1$. Furthermore, the restriction of $\xi _g$ to $\Lambda _{2d}$ \textup{(}or to any other closed, bounded, shift-invariant subset of $\ell ^\infty (\mathbb{Z}^d,\mathbb{Z})$\textup{)} is continuous in the product topology on that space.
	\end{lemm}

	\begin{proof}
We fix $x\in X_{f^{(d)}}$ and define $w\in W_{f^{(d)}}$ by demanding that $\rho (w)=x$ and $0\le w_\mathbf{n}<1$ for every $\mathbf{n}\in\mathbb{Z}^d$. If $v=f^{(d)}(\sigma )(w)$ then $-2d+1\le v_\mathbf{n}\le 2d-1$ for every $\mathbf{n}\in\mathbb{Z}^d$.

Since $\bar{\xi }_g$ commutes with $f^{(d)}(\sigma )$ by \eqref{eq:equivariance}, \eqref{eq:timesg2} shows that
	\begin{equation}
	\label{eq:xigg}
\xi _g(v)=(\rho \circ \bar{\xi }_g)(v)=g(\alpha )(x).
	\end{equation}
Hence
	\begin{equation}
	\label{eq:cover}
X_{f^{(d)}}\supset \xi _g(\ell ^\infty (\mathbb{Z}^d,\mathbb{Z}))\supset \xi _g(V_{2d})\supset g(\alpha )(X_{f^{(d)}}),
	\end{equation}
where $V_K=\{-K+1,\ldots,K-1\}^{\Z^d}\subset \ell^\infty(\Z^d,\Z)$.

We claim that
	\begin{equation}
	\label{eq:g(X)}
g(\alpha )(X_{f^{(d)}})=X_{f^{(d)}}.
	\end{equation}
Indeed, consider the exact sequence
	$$
\smash{\{0\}\longrightarrow \ker g(\alpha )\cap X_{f^{(d)}}\longrightarrow X_{f^{(d)}}\overset{g(\alpha )}\longrightarrow X_{f^{(d)}}\longrightarrow \{0\},}
	$$
set $Y=\ker g(\alpha )\cap X_{f^{(d)}}$, $Z=g(\alpha )(X_{f^{(d)}})\subset X_{f^{(d)}}$, write $\alpha _Y$ and $\alpha _Z$ for the restrictions of $\alpha $ to $Y$ and $Z$, and denote by $\alpha '$ the $\mathbb{Z}^d$-action induced by $\alpha $ on $X_{f^{(d)}}/Z$.

Yuzvinskii's addition formula (\cite[(14.1)]{DSAO}) implies that
	$$
h_\textup{top}(\alpha _{f^{(d)}})=h_\textup{top}(\alpha _Y) \linebreak[0]+h_\textup{top}(\alpha _Z)=h_\textup{top}(\alpha ')+h_\textup{top}(\alpha _Z),
	$$
where we are using the fact that the topological entropies of these actions coincide with their metric entropies with respect to Haar measure. Since the polynomials $f^{(d)}$ and $g$ have no common factors, $h_\textup{top}(\alpha _Y)=0$ by \cite[Corollary 18.5]{DSAO}, hence $h_\textup{top}(\alpha _{f^{(d)}})=h_\textup{top}(\alpha _Z)$ is given by \eqref{eq:entropy} and $0<h_\textup{top}(\alpha _{f^{(d)}})<\infty$. Since the Haar measure $\lambda _{X_{f^{(d)}}}$ of $X_{f^{(d)}}$ is the unique measure of maximal entropy for $\alpha _{f^{(d)}}$ we conclude that $\lambda _{X_{f^{(d)}}}(g(\alpha )(X_{f^{(d)}}))=1$ and $g(\alpha )(X_{f^{(d)}})= X_{f^{(d)}}$, as claimed in \eqref{eq:g(X)}.

We have proved that $\xi _g(V_{2d})=X_{f^{(d)}}$. If $v'\in \ell ^\infty (\mathbb{Z}^d,\mathbb{Z})$ satisfies that $v'_\mathbf{n}=2d-1$ for every $\mathbf{n}\in\mathbb{Z}^d$, then $v'+V_{2d}=\Lambda _{4d-1}$, and \eqref{eq:cover} implies that $\xi _g(\Lambda _{4d-1})=\xi _g(V_{2d})+\xi _g(v')=X_{f^{(d)}}+\xi _g(v')=X_{f^{(d)}}$.

\medskip We still have to show that $\xi _g(\Lambda _{2d})=X_{f^{(d)}}$. Fix $M\ge1$ for the moment and put
	\begin{equation}
	\label{eq:QM}
Q_M=\{-M,\dots ,M\}^d\subset \mathbb{Z}^d.
	\end{equation}
Let
	$$
\ell ^\infty (\mathbb{Z}^d,\mathbb{Z}_+)=\{v\in\ell ^\infty (\mathbb{Z}^d,\mathbb{Z}):v_\mathbf{n}\ge0\enspace \textup{for every}\enspace \mathbf{n}\in\mathbb{Z}^d\}.
	$$
For every $v\in\ell ^\infty (\mathbb{Z}^d,\mathbb{Z}_+)$ and $\mathbf{n}\in\mathbb{Z}^d$ we set
	$$
h^{(v,\mathbf{n})}= \smash[b]{
	\begin{cases}
u^\n\cdot f^{(d)}&\textup{if}\enspace v_\mathbf{n}\ge2d
	\\
0&\textup{otherwise},
	\end{cases}
}
	$$
and we put
	$$
\smash[b]{H^{(v,M)}=\sum_{\mathbf{n}\in Q_M}h^{(v,\mathbf{n})},\enspace \enspace T(v)=v-H^{(v,M)}.}
	$$
If
	\begin{equation}
	\label{eq:D}
\smash[t]{D_M(v)=\sum_{\mathbf{n}\in Q_M}v_\mathbf{n}\cdot \|\mathbf{n}\|_\textup{max}^2,}
	\end{equation}
where $\|\cdot \|_\textup{max}$ is the maximum norm on $\mathbb{R}^d$, then $T(v)=v$ if and only if $v_\mathbf{n}<2d$ for every $\mathbf{n}\in Q_M$, and
	\begin{equation}
	\label{eq:increase}
D_M(T(v))\ge D_M(v)+2
	\end{equation}
otherwise. We define inductively $T^n(v)=T(T^{n-1}(v)),\,n\ge2$, and conclude from \eqref{eq:increase} that there exists, for every $v\in \ell ^\infty (\mathbb{Z}^d,\mathbb{Z}_+)$, an integer $K_M(v)\ge0$ with
	\begin{equation}
	\label{eq:tildev}
\tilde{v}^{(M)}=T^k(v)\enspace \textup{for every}\enspace k\ge K_M(v).
	\end{equation}
For $v\in\Lambda _{4d-1}$ and any $M\ge 1$, the corresponding $\tilde{v}^{(M)}$ satisfies
	\begin{equation}
	\label{eq:range2}
	\begin{gathered}
0\le \tilde{v}^{(M)}_\mathbf{n}\le 2d-1\enspace \textup{if}\enspace \mathbf{n}\in Q_M,
	\\
\tilde{v}^{(M)}_\mathbf{n}\ge v_\mathbf{n}\enspace \textup{if}\enspace \|\mathbf{n}\|_\textup{max}=M+1,
	\\
\sum_{\{\mathbf{n}:\|\mathbf{n}\|_{\textup{max}}=M+1\}} \tilde{v}^{(M)}_\mathbf{n}-v_\mathbf{n}\le (2d-1)\cdot (2M+1)^d,
	\\
\tilde{v}^{(M)}_\mathbf{n}=v_\mathbf{n} \enspace\textup{if}\enspace \|\mathbf{n}\|_\textup{max}>M+1,
	\end{gathered}
	\end{equation}
where $\|\cdot \|_{\textup{max}}$ is the maximum norm on $\mathbb{R}^d$.

Let $\tilde{V}^{(M)}=\{\tilde{v}^{(M)}:v\in \Lambda _{4d-1}\}$. Since $v-\tilde{v}^{(M)}\in (f^{(d)})$ it is clear that $\xi _g(v)=\xi _g(\tilde{v}^{(M)})$ for every $v\in \Lambda _{4d-1}$ and $g\in \tilde{I}_d$.

Since $g\in \tilde{I}_d$, Theorem \ref{t:homoclinic} implies that there exists a constant $C>0$ with
	$$
|(g^*\cdot w^{(d)})_\mathbf{n}|\le C\cdot \|\mathbf{n}\|_\textup{max}^{-d-1}\enspace \textup{for every nonzero}\enspace \mathbf{n}\in\mathbb{Z}^d.
	$$
Hence
	$$
|\bar{\xi }_g(\tilde{v}^{(M)})_\mathbf{0} - \bar{\xi }_g(\bar{v}^{(M)})_\mathbf{0}|<4d\cdot (2M+1)^d\cdot C\cdot (M+1)^{-d-1}\to 0
	$$
as $M\to\infty $, where
	$$
\smash{\bar{v}^{(M)}_\mathbf{n}=
	\begin{cases}
\tilde{v}^{(M)}_\mathbf{n}&\textup{if}\enspace \mathbf{n}\in Q_M,
	\\
v_\mathbf{n}&\textup{otherwise}.
	\end{cases}
}
	$$
It follows that
	$$
\lim_{M\to\infty }\bar{\xi }_g(v-\bar{v}^{(M)})=0
	$$
in the topology of coordinate-wise convergence. Since
	$$
\bar{v}^{(M)}\in\{v\in\Lambda _{4d-1}:0\le v_\mathbf{n}<2d\enspace \textup{for every}\enspace \mathbf{n}\in Q_M\}
	$$
for every $v\in \Lambda _{4d-1}$ and $M\ge1$, we conclude that $\xi _g(\Lambda _{2d})$ is dense in  $X_{f^{(d)}}$. As $\xi _g(\Lambda _{2d})$ is also closed, this implies that $\xi _g(\Lambda _{2d})=X_{f^{(d)}}$, as claimed.
	\end{proof}

	\begin{rema}
	\label{r:surjective2}
Although we have not yet introduced sandpiles and their stabilization (this will happen in Section \ref{s:sandpiles}), the second part of the proof of Lemma \ref{l:surjective} is effectively a `sandpile' argument, and $\tilde{v}^{(M)}$ is a stabilization of $v$ in $Q_M$.
	\end{rema}

	\begin{proof}[Proof of Proposition \ref{p:xig}]
If $g$ lies in $\tilde{I}_d$, Lemma \ref{l:surjective} shows that $\xi _g(\Lambda _{2d})=\xi _g(\ell ^\infty (\mathbb{Z}^d,\mathbb{Z}))=X_{f^{(d)}}$. On the other hand, if $g=h\cdot f^{(d)}$ for some $h\in R_d$, then $g^*\cdot w^{(d)}\in R_d$, and hence $\bar\xi_g(v)_\n\in\Z$ for every $\n\in\Z^d$ and $v\in\ell^\infty(\Z^d,\Z)$, implying that $\xi_g(v)=0$.
	\end{proof}

\subsection{\label{ss:kernels}Kernels of covering maps} Having found compact shift-invariant subsets $V\subset \ell ^\infty (\mathbb{Z}^d,\mathbb{Z})$ such that the restrictions of $\xi _g$ to $V$ are surjective for every $g\in\tilde{I}_d$ (cf. Lemma \ref{l:surjective}), we turn to the problem of determining the kernels of the group homomorphisms $\xi _g\colon \ell ^\infty (\mathbb{Z}^d,\mathbb{Z})\longrightarrow X_{f^{(d)}},\,g\in I_d$ (cf. \eqref{eq:barxi}). We shall see below that $\ker(\xi _g)$ depends on $g$ and that $\ker\xi _{gh}\supsetneq \ker(\xi _g)$ for $g\in I_d$ and $0\ne h\in R_d$. In view of this it is desirable to characterize the set
	\begin{equation}
	\label{eq:kerI}
K_d=\bigcap_{g\in I_d}\ker(\xi _g)
	\end{equation}
of all $v\in\ell ^\infty (\mathbb{Z}^d,\mathbb{Z})$ which are sent to $0$ by \textit{every} $\xi _g,\,g\in I_d$.

In the following discussion we set, for every ideal $J\subset R_d$,
	\begin{equation}
	\label{eq:XJ}
\smash[b]{X_J=\{x\in \mathbb{T}^{\mathbb{Z}^d}:g(\alpha )( x)=0\enspace \textup{for every}\enspace g\in J\}=\bigcap_{g\in J}\ker g(\alpha ),}
	\end{equation}
and put
	\begin{equation}
	\label{eq:tildeX}
\tilde{X}_{f^{(d)}}= \widehat{\mathscr{I}_d/(f^{(d)})}=X_{f^{(d)}}/X_{I_d}.
	\end{equation}
In order to explain \eqref{eq:tildeX} we note that the dual group of $\tilde{X}_{f^{(d)}}$ is a subgroup of $\widehat{X_{f^{(d)}}}=R_d/(f^{(d)})$, hence $\tilde{X}_{f^{(d)}}$ is a quotient of $X_{f^{(d)}}$ by a closed, shift-invariant subgroup, which is the annihilator of $I_d/(f^{(d)})$ and hence equal to $X_{I_d}$. The $\mathbb{Z}^d$-action $\alpha _{f^{(d)}}$ on $X_{f^{(d)}}$ induces a $\mathbb{Z}^d$-action $\tilde{\alpha }_{f^{(d)}}$ on $\tilde{X}_{f^{(d)}}$. Note that $\tilde{\alpha }_{f^{(d)}}^\mathbf{n}$ is dual to multiplication by $u^\mathbf{n}$ on $\mathscr{I}_d/(f^{(d)})$. With this notation we have the following result.

	\begin{theo}
	\label{t:kernels}
There exists a surjective group homomorphisms $\eta \colon \ell ^\infty (\mathbb{Z}^d,\mathbb{Z})\longrightarrow \tilde{X}_{f^{(d)}}$ with the following properties.
	\begin{enumerate}
	\item
The homomorphism $\eta $ is equivariant in the sense that $\eta \circ \sigma ^\mathbf{n}=\tilde{\alpha }_{f^{(d)}}^\mathbf{n}\circ \eta $ for every $\mathbf{n}\in\mathbb{Z}^d$;
	\item
$\ker(\eta )=K_d$;
	\item
The topological entropy of $\tilde{\alpha }_{f^{(d)}}$ coincides with that of $\alpha _{f^{(d)}}$ \textup{(}cf. \eqref{eq:entropy}\textup{)}.
	\end{enumerate}
	\end{theo}

For the proof of Theorem \ref{t:kernels} we choose and fix a set of generators $G_d=\{g^{(1)},\dots ,\linebreak[0]g^{(m)}\}$ of $I_d$ (for $d=2$ we may take, for example, $G_2=\{g^{(1)},g^{(2)},g^{(3)}\}$ with $g^{(1)}=(1-u_1)^2\cdot (1-u_2)$, $g^{(2)}=(1-u_1)\cdot (1-u_2)^2$ and $g^{(3)}=(1-u_1)^2+(1-u_2)^2)$; for $d\ge3$ we can use the set of generators $G_d=\{f^{(d)}\}\cup\{(u_i-1)\cdot (u_j-1)\cdot (u_k-1):i,j,k=1,\dots ,d\}$). With such a choice of $G_d$ we define a map
	\begin{equation}
	\label{eq:xiI1}
\xi _{I_d}\colon \ell ^\infty (\mathbb{Z}^d,\mathbb{Z})\longrightarrow X_{f^{(d)}}^m
	\end{equation}
by setting
	\begin{equation}
	\label{eq:xiId}
\xi _{I_d}(v)=(\xi _{g^{(1)}}(v),\dots ,\xi _{g^{(m)}}(v))
	\end{equation}
for every $v\in \ell ^\infty (\mathbb{Z}^d,\mathbb{Z})$.

	\begin{lemm}
	\label{l:xiI}
There exists a continuous shift-equivariant group isomorphism
	\begin{equation}
	\label{eq:theta}
\theta _d\colon \xi _{I_d}(\ell ^\infty (\mathbb{Z}^d,\mathbb{Z}))\longrightarrow \tilde{X}_{f^{(d)}}.
	\end{equation}
	\end{lemm}

	\begin{proof}
We define a continuous group homomorphism $\theta '\colon X_{f^{(d)}}\longrightarrow X_{f^{(d)}}^m$ by setting $\theta '(x)=(g^{(1)}(\alpha )(x),\linebreak[0]\dots ,g^{(m)}(\alpha )(x))$ for every $x\in X_{f^{(d)}}$.

According to \eqref{eq:barxi} and \eqref{eq:equivariance},
	$$
\xi _g\circ h(\sigma )(v)=\rho (g^*\cdot w^{(d)}\cdot h^*\cdot v)=g(\alpha )\circ \xi _h(v)
	$$
for every every $g,h\in\tilde{I}_d$ and $v\in R_d$, and hence, by continuity, for every $g,h\in\tilde{I}_d$ and $v\in\ell ^\infty (\mathbb{Z}^d,\mathbb{Z})$. Since $\xi _h(\ell ^\infty (\mathbb{Z}^d,\mathbb{Z}))=X_{f^{(d)}}$ by Lemma \ref{l:surjective} we conclude that
	$$
\xi _{I_d}(\ell ^\infty (\mathbb{Z}^d,\mathbb{Z}))\supset \xi _{I_d}\circ h(\sigma )(\ell ^\infty (\mathbb{Z}^d,\mathbb{Z}))=\theta '(X_{f^{(d)}}).
	$$
On the other hand,
	$$
\xi _{I_d}(v)=(g^{(1)}(\alpha )\circ v^*(\alpha )(x^\Delta ),\dots ,g^{(m)}(\alpha )\circ v^*(\alpha )(x^\Delta ))\in \theta '(X_{f^{(d)}})
	$$
for every $v\in R_d$ and hence, again by continuity, for every $v\in\ell ^\infty (\mathbb{Z}^d,\mathbb{Z})$. We have proved that
	$$
\xi _{I_d}(\ell ^\infty (\mathbb{Z}^d,\mathbb{Z})))=\theta '(X_{f^{(d)}}).
	$$

The homomorphism $\theta '$ has kernel $X_{I_d}$ and induces a group isomorphism $\theta ''\colon \tilde{X}_{f^{(d)}}\longrightarrow \linebreak[0]\theta ' (X_{f^{(d)}})$. The proof is completed by setting $\theta _d=(\theta '')^{-1}$.
	\end{proof}

	\begin{proof}[Proof of Theorem \ref{t:kernels}]
We set $\eta =\theta _d\circ \xi _{I_d}$ (cf. \eqref{eq:xiI1}--\eqref{eq:theta}). By definition, $K_d=\ker(\xi _{I_d})\linebreak[0]=\ker(\eta )$.

The equivariance of $\eta $ is obvious. Furthermore, $h_{\textup{top}}(\tilde{\alpha })\le h_{\textup{top}}(\alpha _{f^{(d)}})$, since $\tilde{X}_{f^{(d)}}$ is an equivariant quotient of $X_{f^{(d)}}$. On the other hand, $\tilde{X}_{f^{(d)}}\cong \xi _{I_d}(\ell ^\infty (\mathbb{Z}^d,\mathbb{Z}))$, and the first coordinate projection $\pi _1\colon \xi _{I_d}(\ell ^\infty (\mathbb{Z}^d,\mathbb{Z})\longrightarrow X_{f^{(d)}}$ is surjective by Lemma \ref{l:surjective}. This implies that $h_{\textup{top}}(\beta _1)=h_{\textup{top}}(\tilde{\alpha })\ge h_{\textup{top}}(\alpha _{f^{(d)}})$, so that these entropies have to coincide.
	\end{proof}

In order to characterize the kernel $K_d$ of $\eta $ further we need a lemma and a definition.

	\begin{lemm}
	\label{l:cg}
For every $y\in\ell ^\infty (\mathbb{Z}^d)$ with $\rho (y)\in X_{\mathscr{I}_d^3}$ there exists a unique $c(y)\in[0,1)$ with $f^{(d)}\cdot y+\tilde{c}(y)\in\ell ^\infty (\mathbb{Z}^d,\mathbb{Z})$, where $\tilde{c}(y)$ denotes the element of $\widetilde{\mathbb{R}}$ with $\tilde{c}(y)_\mathbf{n}=c(y)$ for every $\mathbf{n}\in\mathbb{Z}^d$.
	\end{lemm}

	\begin{proof}
Let $x\in X_{\mathscr{I}_d^3}$ and $y\in\ell ^\infty (\mathbb{Z}^d)$ with $\rho (y)=x$. According to the definition of $X_{\mathscr{I}_d^3}$ this means that
	$$
g(\alpha )(x)=\rho (g^*\cdot y)=0
	$$
for every $g\in \mathscr{I}_d^3$.

Since $g_j=(u_j-1)\cdot f^{(d)}\in\mathscr{I}_d^3$ for $j=1,\dots ,d$, $g_j(\alpha )(x)=\rho (g_j^*\cdot y)=0$ for $j=1,\dots ,d$, which implies that $f^{(d)}(\alpha )(x)$ is a fixed point of the $\mathbb{Z}^d$-action $\alpha $ on $X_{\mathscr{I}_d^3}$. Hence there exists a unique constant $c(y)\in[0,1)$ with $f^{(d)}\cdot y+\tilde{c}(y)\in\ell ^\infty (\mathbb{Z}^d,\mathbb{Z})$.
	\end{proof}

	\begin{defi}
	\label{d:periodic}
We call points $v\in \ell ^\infty (\mathbb{Z}^d)$ and $x\in\mathbb{T}^{\mathbb{Z}^d}$ \textit{periodic} if their orbits (under $\sigma $ and $\alpha $, respectively) are finite.

If $\Gamma \subset \mathbb{Z}^d$ is a subgroup of finite index we denote by $\ell (\mathbb{Z}^d)^{(\Gamma )}$, $\ell ^\infty (\mathbb{Z}^d,\mathbb{Z})^{(\Gamma )}$ and $K_d^{(\Gamma )}$ the sets of all $\Gamma $-invariant elements in the respective spaces.
	\end{defi}

	\begin{theo}
	\label{t:periodic}
\textup{(1)} For every $y\in\ell ^\infty (\mathbb{Z}^d)$ with $\rho(y)\in X_{\mathscr{I}^3_d}$,
	\begin{equation}
	\label{eq:period}
v=f^{(d)}\cdot y+\tilde{c}(y)+\tilde{m}\in K_d\subset \ell ^\infty (\mathbb{Z}^d,\mathbb{Z})
	\end{equation}
for every $\tilde{m}\in \widetilde{\mathbb{Z}}$ \textup{(}cf. \eqref{eq:Idform}, \eqref{eq:widetilde}, \eqref{eq:XJ} and Lemma \ref{l:cg}\textup{)}.

\smallskip \textup{(2)} Let $\Gamma \subset \mathbb{Z}^d$ be a subgroup of finite index. An element $v\in\ell ^\infty (\mathbb{Z}^d,\mathbb{Z})^{(\Gamma )}$ lies in $K_d$ if and only if it is of the form \eqref{eq:period} with $y\in\ell ^\infty (\mathbb{Z}^d)^{(\Gamma )}$, $\rho(y)\in X_{\mathscr{I}^3_d}$ and $\tilde{m}\in \widetilde{\mathbb{Z}}$.
	\end{theo}

We start the proof of Theorem \ref{t:periodic} with two lemmas.

	\begin{lemm}
	\label{l:constant}
For every $g\in I_d$ and every constant element $\tilde{m}\in\ell ^\infty (\mathbb{Z}^d,\mathbb{Z})$, $\xi _g(\tilde{m})=0$. In other words, $\widetilde{\mathbb{Z}}\subset K_d$.
	\end{lemm}

	\begin{proof}
We know that $g\in I_d$ if and only if it satisfies \eqref{eq:condA}--\eqref{eq:condD}. We fix $g=\sum_{\mathbf{k}\in\mathbb{Z}^d} g_\mathbf{k}u^\mathbf{k}\linebreak[0]\in I_d$, put $v=g^*\cdot w^{(d)}\in\ell ^1(\mathbb{Z}^d)$, and set $c=\sum_{\mathbf{k}\in\mathbb{Z}^d} g_\mathbf{k} k_j^2\in\mathbb Z$ (note that this value is independent of $j\in\{1,\dots ,d\}$ by \eqref{eq:condD}).

For every $\mathbf n\in{\mathbb Z}^d$,
	$$
v_\mathbf{n}=(g^*\cdot w^{(d)})_\mathbf{n} =\sum_{\mathbf{k}\in\mathbb{Z}^d} g_\mathbf{k} w_{\mathbf{n}+\mathbf{k}}^{(d)} =\int_{\mathbb{T}^d} e^{-2\pi i \langle \mathbf{n},\mathbf{t}\rangle }\frac {\sum_{\mathbf{k}} g_\mathbf{k} e^{-2\pi i\langle \mathbf{k},\mathbf{t}\rangle }} {2d-2\sum_{j=1}^d\cos(2\pi t_j)} \,d\mathbf{t}.
	$$
Hence $v=(v_\mathbf{n})$ is the sequence of Fourier coefficients of the function
	$$
H_g(\mathbf{t}) = \frac {\sum_{\mathbf{k}} g_\mathbf{k} e^{-2\pi i\langle \mathbf{k},\mathbf{t}\rangle }} {2d-2\sum_{j=1}^d\cos(2\pi t_j)}.
	$$
Since these Fourier coefficients are absolutely summable by assumption, we get that
	\begin{equation}
	\label{eq:coeffsum}
\sum_{\mathbf n\in\mathbb{Z}^d} v_{\mathbf n} = H_g(0).
	\end{equation}
On the other hand, given the Taylor series expansion of $H_g$ at $\mathbf t=\mathbf 0$, we have
	$$
\smash[b]{H_g(t)=\frac {-2\pi^2 \sum_{j=1}^d t_j^2 \bigl( \sum_{\mathbf k} g_\mathbf{k} k_j^2\bigr)+\text{h.o.t}}
{4\pi^2 \sum_{j=1}^d t_j^2 +\text{h.o.t}}= \frac {-2\pi^2 c \sum_{j=1}^d t_j^2+\text{h.o.t}}{4\pi^2 \sum_{j=1}^d t_j^2 +\text{h.o.t}}}
	$$
and hence
	$$
H_g(0)= -c/2.
	$$
We are going to show that $H_g(0)\in\mathbb Z$. Indeed, since $\sum_{\mathbf k} g_{\mathbf k} k_j=0$ for all $j$ by \eqref{eq:condB}, we have that
	\begin{equation}
	\label{eq:Hg(0)}
H_g(0) =-\frac 12\sum_{\mathbf k} g_\mathbf{k} k_j^2 = -\frac 12\sum_{\mathbf k} g_\mathbf{k} k_j (k_j-1)=
-\sum_{\mathbf k} g_\mathbf{k} \frac{ k_j(k_j-1)}2\in\mathbb Z.
	\end{equation}
Finally, for any $g\in I_d$ and $\tilde m\in \widetilde {\mathbb Z}$, we have
	\begin{equation}
	\label{eq:multiple}
\smash{\bar\xi_g(\tilde m) =m\cdot \sum_{\mathbf n\in\mathbb{Z}^d} v_{\mathbf n} =mH_g(0)\in \mathbb Z}
	\end{equation}
by \eqref{eq:coeffsum}, and hence $\xi_g(\tilde m)=0\in X_{f^{(d)}}$.
	\end{proof}

	\begin{lemm}
	\label{l:zero}
For every $g\in\mathscr{I}_d^3$, $H_g(0)=0$ \textup{(}cf. \eqref{eq:coeffsum}\textup{)}.
	\end{lemm}

	\begin{proof}
Every element of $\mathscr{I}_d^3$ is of the form $h\cdot g$ with $h\in R_d$ and $g=(u_i-1)\cdot (u_j-1)\cdot (u_k-1)$ for some $i,j,k\in\{1,\dots ,d\}$. We set $v=g^*\cdot w^{(d)}$ and obtain from \eqref{eq:Hg(0)} that $H_g(0)=\sum_{\mathbf{n}\in\mathbb{Z}^d}v_\mathbf{n}=0$. If $w=(hg)^*\cdot w^{(d)}=h^*\cdot v$, then $H_{hg}(0)=\sum_{\mathbf{n}\in\mathbb{Z}^d}w_\mathbf{n}=\sum_{\mathbf{k}\in\mathbb{Z}^d}h_\mathbf{k}\sum_{\mathbf{n}\in\mathbb{Z}^d}v_{\mathbf{n-k}}=0$.
	\end{proof}

	\begin{proof}[Proof of Theorem \ref{t:periodic}]
Let $x\in X_{\mathscr{I}_d^3}$, $y\in\ell ^\infty (\mathbb{Z}^d)$ with $\rho (y)=x$, $\tilde{m}\in\widetilde{\mathbb{Z}}$, and $v=f^{(d)}\cdot y+\tilde{c}(y)+\tilde{m}\in\ell ^\infty (\mathbb{Z}^d,\mathbb{Z})$ (cf. Lemma \ref{l:cg}). Then
	$$
g(\alpha )(x)=\rho (g^*\cdot y)=0
	$$
for every $g\in \mathscr{I}_d^3$. We set $w=g^*\cdot w^{(d)}$ and obtain from \eqref{eq:equivariance}, \eqref{eq:lemmaWf} and Lemma \ref{l:constant}, that
	\begin{align*}
\xi _g(v)&=\xi _g(f^{(d)}\cdot y+\tilde{c}(y)+\tilde{m})=\xi _g(f^{(d)}\cdot y+\tilde{c}(y))=
	\\
&=\rho (g^*\cdot w^{(d)}\cdot f^{(d)}\cdot y + g^*\cdot w^{(d)}\cdot \tilde{c}(y))=\rho (g^*\cdot y + w\cdot \tilde{c}(y))
	\\
&=\rho (g^*\cdot y)=0,
	\end{align*}
since $\sum_{\mathbf{n}\in\mathbb{Z}^d}w_\mathbf{n}=0$ by Lemma \ref{l:zero}. This proves that every $v\in\ell ^\infty (\mathbb{Z}^d,\mathbb{Z})$ of the form \eqref{eq:period} lies in $K_d$.

For (2) we assume that $\Gamma \subset \mathbb{Z}^d$ is a subgroup of finite index. In view of (1) we only have to verify that every $v\in \ell ^\infty (\mathbb{Z}^d,\mathbb{Z})^{(\Gamma )}\cap K_d$ has the form \eqref{eq:period}.

Assume therefore that $v\in \ell ^\infty (\mathbb{Z}^d,\mathbb{Z})^{(\Gamma )}\cap K_d$. We choose a set $C_\Gamma \subset \mathbb{Z}^d$ which intersects each coset of $\Gamma $ in $\mathbb{Z}^d$ in a single point and set $\ell _0^{(\Gamma )}=\{w\in\ell ^\infty (\mathbb{Z}^d)^{(\Gamma )}:\sum_{\mathbf{n}\in C_\Gamma }w_\mathbf{n}=0\}$. As $\ell _0^{(\Gamma )}$ is finite-dimensional and $\ker(f^{(d)}(\sigma ))=\widetilde{\mathbb{R}}$ there exists, for every $y\in \ell _0^{(\Gamma )}$, a unique $y'\in\ell _0^{(\Gamma )}$ with $f^{(d)}\cdot y'=y$.

Put $\tilde{a}=\bigl(\sum_{\mathbf{n}\in C_\Gamma }v_\mathbf{n}\bigr)/|\mathbb{Z}^d/\Gamma |$, regarded as an element of $\widetilde{\mathbb{R}}$. If $v'=v-\tilde{a}$, then $v'\in\ell _0^{(\Gamma )}$ and $f^{(d)}\cdot y=v'$ for some $y\in\ell _0^{(\Gamma )}$.

Since $v\in K_d$, $\xi _g(v)=0$ for every $g\in I_d$. For $g\in\mathscr{I}_d^3$, Lemma \ref{l:zero} shows that
	\begin{align*}
\bar{\xi }_g(v)&=g^*\cdot w^{(d)}\cdot v=g^*\cdot w^{(d)}\cdot v' + g^*\cdot w^{(d)}\cdot \tilde{a}= g^*\cdot y + g^*\cdot w^{(d)}\cdot \tilde{a}
	\\
&= g^*\cdot y\in\ell ^\infty (\mathbb{Z}^d,\mathbb{Z}).
	\end{align*}
Hence $\rho (g^*\cdot y)=g(\alpha )(\rho (y))=0$ for all $g\in\mathscr{I}_d^3$, so that $\rho (y)\in X_{\mathscr{I}_d^3}$.

We obtain that
	$$
v=f^{(d)}\cdot y+\tilde{a}
	$$
for some $y\in\ell ^\infty (\mathbb{Z}^d)$ with $\rho (y)\in X_{\mathscr{I}_d^3}$ and some $\tilde{a}\in\widetilde{\mathbb{R}}$, which completes the proof of (2).
	\end{proof}

Theorem \ref{t:periodic} implies that there exist nonconstant elements $v\in K_d\smallsetminus f^{(d)}(\sigma )(\ell ^\infty (\mathbb{Z}^d,\mathbb{Z}))$. However, if two elements $v,v'\in \ell ^\infty (\mathbb{Z}^d,\mathbb{Z})$ differ in only finitely many coordinates, then they get identified under $\xi _{I_d}$ (i.e., their difference lies in $K_d$) if and only if they differ by an element in $(f^{(d)})\subset \ell ^\infty (\mathbb{Z}^d,\mathbb{Z})$. This is a consequence of the following assertion.

	\begin{prop}
	\label{p:injective}
For every $g\in \tilde{I}_d$, $\ker(\xi _g)\cap R_d=(f^{(d)})=f^{(d)}\cdot R_d$.
	\end{prop}

	\begin{proof}
Suppose that $h\in R_d\cap \ker(\xi _g)$. Then
	\begin{equation}
	\label{eq:injective}
v\defeq \bar{\xi }_g(h)=g^*\cdot w^{(d)}\cdot h\in\ell ^\infty (\mathbb{Z}^d,\mathbb{Z}).
	\end{equation}
Since $g\in I_d$, $g^*\cdot w^{(d)}\in \ell ^1(\mathbb{Z}^d)$ and hence $v\in R_d=\ell ^1(\mathbb{Z}^d)\cap \ell ^\infty (\mathbb{Z}^d,\mathbb{Z})$. If we multiply both sides of \eqref{eq:injective} by $f^{(d)}$ we get that
	$$
f^{(d)}\cdot v=g\cdot h.
	$$
As $R_d$ has unique factorization this implies that $h\in f^{(d)}\cdot R_d$.
	\end{proof}

	\begin{rems}
	\label{r:periodic}
(1) One can show that the periodic points are dense in $K_d$, so that every $v\in K_d$ is a coordinate-wise limit of elements of the form \eqref{eq:period} in Theorem \ref{t:periodic}.

\smallskip (2) Theorem \ref{t:periodic} (1) gives a `lower bound' for the kernel $K_d$ of the maps $\xi _g,\,g\in I_d$. There is also a straightforward `upper bound' for that kernel: an element $v\in \ell ^\infty (\mathbb{Z}^d,\mathbb{Z})$ lies in $K_d$ if and only if
	$$
\bar{\xi }_g(v)=g^*\cdot w^{(d)}\cdot v\eqdef w_g\in\ell ^\infty (\mathbb{Z}^d,\mathbb{Z})\enspace \textup{for every}\enspace g\in I_d.
	$$
By multiplying this equation with $f^{(d)}$ we obtain that
	\begin{equation}
	\label{eq:kernel2}
K_d\subset \{v\in\ell ^\infty (\mathbb{Z}^d,\mathbb{Z}):g\cdot v\in f^{(d)}\cdot \ell ^\infty (\mathbb{Z}^d,\mathbb{Z})\enspace \textup{for every}\enspace g\in I_d\}\eqdef\bar{K}_d.
	\end{equation}
It is not very difficult to see that the inclusion in \eqref{eq:kernel2} is strict. In fact, $\bar{K}_d/K_d$ turns out to be isomorphic to $\mathbb{T}^d$.

\smallskip (3) In \cite{vdPT}, the kernel $K_d$ of $\xi_{I_d}$ was studied using methods of commutative algebra.
	\end{rems}

\section{The abelian sandpile model}\label{s:sandpiles}

Let $d\ge2$, $\gamma \ge 2d$, and let $E\subset \mathbb{Z}^d$ be a nonempty set. For every $\mathbf{n}\in E$ we denote by $\mathsf{N}_E(\mathbf{n})$ the number of neighbours of $\mathbf{n}$ in $E$, i.e.,
	\begin{equation}
	\label{eq:neighbours}
\mathsf{N}_E(\mathbf{n})=\bigl|E\cap\{\mathbf{n}\pm\mathbf{e}^{(i)}:1=1,\dots ,d\}\bigr|,
	\end{equation}
where $\mathbf{e}^{(i)}$ is the $i$-th unit vector in $\mathbb{Z}^d$. We set
	\begin{equation}
	\label{eq:Lambda}
\smash{\Lambda _\gamma =\{0,\dots ,\gamma -1\}^{\mathbb{Z}^d}}
	\end{equation}
(cf. Lemma \ref{l:surjective}) and put
	\begin{equation}
	\label{eq:RE}
	\begin{gathered}
\mathcal{P}_E^{(\gamma )}=\{v\in\{0,\dots ,\gamma -1\}^E:v_\mathbf{n}\ge \mathsf{N}_E(\mathbf{n})\enspace \textup{for at least one}\enspace \mathbf{n}\in E\},
	\\
\mathcal{R}_E^{(\gamma )}= \bigcap_{\substack{\varnothing \ne F\subset E
	\\
0<|F|<\infty }}\mathcal{P}_F.
	\end{gathered}
	\end{equation}
In the literature the set $\mathcal{R}_E^{(\gamma )}$ is called the set of \textit{recurrent configurations} on $E$. A configuration $v\in\{0,\ldots,\gamma -1\}^E$ is recurrent if and only if it passes the \textit{burning test}, which is described as follows: given $v\in \{0,\ldots,\gamma -1\}^E$, delete (or \textit{burn}) all sites $\mathbf{n}\in E$ such that
	$$
v_\mathbf{n} \ge N_E(\mathbf{n}),
	$$
thereby obtaining a configuration $v'\in \{0,\dots,\gamma -1\}^{E^{(1)}}$ with $E^{(1)}\subset E$. We repeat the process and obtain a sequence $E\supset E^{(1)}\supset\dots \supset E^{(k)}\supset \cdots$. If at some stage $E^{(k)}=E^{(k+1)}\ne\varnothing $ we say that $v$ \textit{fails} the burning test, and $v$ is a \textit{forbidden} (or \textit{nonrecurrent}) configuration.

The closed, shift-invariant subset
	\begin{equation}
	\label{eq:sandpiles}
\mathcal{R}_\infty ^{(\gamma )} =\mathcal{R}_{\mathbb{Z}^d}^{(\gamma )}\subset \Lambda _\gamma \subset \ell ^\infty (\mathbb{Z}^d,\mathbb{Z})
	\end{equation}
is called \textit{the $d$-dimensional sandpile model} with parameter $\gamma $. For $\gamma =2d$, $\mathcal{R}_\infty =\mathcal{R}_\infty ^{(2d)}$ is called the \textit{critical sandpile model}, and for $\gamma >2d$, the model $\mathcal{R}_\infty ^{(\gamma )}$ is said to be \textit{dissipative}.

In order to motivate this terminology we assume that $E\subset \mathbb{Z}^d$ is a nonempty set. An element $v\in \mathbb{Z}^E_+$ is called \textit{stable} if $y_\mathbf{n}<\gamma $ for every $\mathbf{n}\in E$. If $v\in\mathbb{Z}^E_+$ is unstable at some $\mathbf{n}\in E$, i.e., if $v_\mathbf{n}\ge\gamma$, then $v$ \textit{topples} at this site: the result is a configuration $T_\mathbf{n}(v)$ with
	$$
T_\mathbf{n}(v)_\mathbf{k}=
	\begin{cases}
v_\mathbf{n}-\gamma &\textup{if}\enspace \mathbf{k}=\mathbf{n},
	\\
v_\mathbf{k}+1 &\textup{if}\enspace \|\mathbf{n}-\mathbf{k}\|_\textup{max}=1,
	\\
v_\mathbf{k} &\textup{otherwise.}
	\end{cases}
	$$
If $v_\mathbf{m},v_\mathbf{m}\ge\gamma$ for some $\mathbf{m},\mathbf{n}\in E$, $\mathbf{m}\ne \mathbf{n}$, then $T_\mathbf{n}( T_\mathbf{n}(v))=T_\mathbf{m}( T_\mathbf{n}(v))$, i.e., toppling operators commute. A stable configuration $\tilde{v}\in\{0,\dots ,\gamma -1\}^E$ is the result of toppling of $v$, if there exist $\mathbf{n}^{(1)},\dots,\mathbf{n}^{(k)} \in E$ such that
	$$
\tilde{v}= \biggl(\prod_{i=1}^k T_{\mathbf{n}^{(i)}}\biggr)(v).
	$$
If the set $E$ is finite (in symbols: $E\Subset \mathbb{Z}^d$), then every $v\in \mathbb{Z}_+^E$ will lead to a stable configuration $\tilde{v}$ by repeated topplings. However, if $E$ is infinite, then repeated toppling of a configuration $v\in\mathbb{Z}_+^E$ will, in general, lead to a stable configuration $\tilde{v}\in\{0,\dots ,\gamma -1\}^E$ only if $\gamma >2d$, i.e., in the dissipative case.\footnote{Even in the dissipative case stable configurations will, in general, only arise as a coordinate-wise limits of infinite sequences of topplings of $v$.}

We denote by $\sigma =\sigma _{\mathcal{R}_\infty ^{(\gamma )}}$ the shift-action of $\mathbb{Z}^d$ on $\mathcal{R}_\infty ^{(\gamma )}\subset \ell ^\infty (\mathbb{Z}^d,\mathbb{Z})\subset W_d$ \textup{(}cf. \eqref{eq:sigma}\textup{)}.

For the following discussion we introduce the Laurent polynomial
	\begin{equation}
	\label{eq:fdgamma}
f^{(d,\gamma )} = \gamma -\sum_{i=1}^{d}(u_i+u_i^{-1})\in R_d=\mathbb{Z}[u_1^\pm,\dots,u_d^\pm].
	\end{equation}
For $\gamma =2d$, $f^{(d,\gamma )}=f^{(d)}$ (cf. \eqref{eq:fd}).

	\begin{prop}
	\label{p:sand}
Let $d\ge2$ and $\gamma \ge 2d$. The following conditions are equivalent for every $v\in\Lambda _\gamma $.
	\begin{enumerate}
	\item
$v\in \mathcal{R}_\infty ^{(\gamma )}$;
	\item
For every nonzero $h\in R_d$ with $h_\mathbf{n}\in\{0,1\}$ for every $\mathbf{n}\in\mathbb{Z}^d$, $(f^{(d,\gamma )}\cdot h)_\mathbf{n}+v_\mathbf{n}\ge \gamma $ for at least one $\mathbf{n}\in\textup{supp}(h)=\{\mathbf{m}\in\mathbb{Z}^d:h_\mathbf{m}\ne0\}$.
	\item
For every $h\in R_d$ with $h_\mathbf{n}>0$ for some $\mathbf{n}\in\mathbb{Z}^d$, $(f^{(d,\gamma )}\cdot h)_\mathbf{n}+v_\mathbf{n}\ge \gamma $ for at least one $\mathbf{n}\in\{\mathbf{m}\in\mathbb{Z}^d:h_\mathbf{m} > 0\}$.
	\end{enumerate}
Furthermore, if $v,v'\in\mathcal{R}_\infty ^{(\gamma )}$ and $0\ne v-v'\in R_d$, then $v-v'\notin f^{(d,\gamma )}\cdot R_d$.
	\end{prop}

	\begin{proof}
Fix an element $v\in \Lambda _\gamma $. If $h\in R_d$ with $h_\mathbf{n}\in\{0,1\}$ for every $\mathbf{n}\in\mathbb{Z}^d$ and $E=\textup{supp}(h)$, then $(f^{(d,\gamma )}\cdot h)_\mathbf{n}+v_\mathbf{n}\in \{0,\dots ,\gamma -1\}$ for every $\mathbf{n}\in E$ if and only if $v_\mathbf{n}\le \mathsf{N}_E(\mathbf{n})-1$ for every $\mathbf{n}\in E$, in which case $\pi _E(v)\notin \mathcal{P}_E$ and $v\notin \mathcal{R}_\infty ^{(\gamma )}$ (cf. \eqref{eq:RE}). This proves the equivalence of (1) and (2).

\smallskip Now suppose that $h\in \ell ^\infty (\mathbb{Z}^d,\mathbb{Z})$ with $M_h= \max_{\mathbf{m}\in\mathbb{Z}^d}h_\mathbf{m}>0$, and that $f^{(d,\gamma )}\cdot h+v\in\Lambda _\gamma $. We set
	\begin{equation}
	\label{eq:Smax}
\mathcal{S}_\textup{max}(h)= \{\mathbf{n}\in\mathbb{Z}^d:h_\mathbf{n}=M_h\}
	\end{equation}
and observe that
	$$
v_\mathbf{n}+(f^{(d,\gamma )}\cdot h)_\mathbf{n}\ge v_\mathbf{n} + M_h\cdot (\gamma - \mathsf{N}_{\mathcal{S}_\textup{max}(h)})<\gamma
	$$
for every $\mathbf{n}\in \mathcal{S}_\textup{max}(h)$, so that
	\begin{equation}
	\label{eq:violation}
v_\mathbf{n}\le \mathsf{N}_{\mathcal{S}_\textup{max}(h)}-1\enspace \textup{for every}\enspace \mathbf{n}\in \mathcal{S}_\textup{max}(h).
	\end{equation}

If $h\in R_d$, then $\mathcal{S}_{\textup{max}}(h)$ is finite and \eqref{eq:violation} yields a contradiction to the definition of $\mathcal{R}_\infty ^{(\gamma )}$. This proves the implication (1) $\Rightarrow$ (3), and the reverse implication (3) $\Rightarrow$ (2) is obvious.

The last assertion of this proposition is a consequence of (3).
	\end{proof}

The proof of Proposition \ref{p:sand} has the following corollary.

	\begin{coro}
	\label{c:sand}
If $v\in \mathcal{R}_\infty ^{(\gamma )}$, and if $h\in \ell ^\infty (\mathbb{Z}^d,\mathbb{Z})$ satisfies that $\max_{\mathbf{m}\in\mathbb{Z}^d}h_\mathbf{m}>0$ and $v+f^{(d,\gamma)}\cdot h\in \mathcal{R}_\infty ^{(\gamma )}$, then every connected\,\footnote{\label{connected}A set $S\subset \mathbb{Z}^d$ is \textit{connected} if we can find, for any two coordinates $\mathbf{m}$ and $\mathbf{n}$ in $S$, a `path' $p(0)=\mathbf{m},\linebreak[0]p(1),\linebreak[0]\dots ,\linebreak[0]p(k)=\mathbf{n}$ in $S$ with $\|p(j)-p(j-1)\|_\textup{max}=1$ for every $j=1,\dots ,k$.} component of $\mathcal{S}_{\textup{max}}(h)$ is infinite \textup{(}cf. \eqref{eq:Smax}\textup{)}.
	\end{coro}

	\begin{proof}
If $\mathcal{S}_\textup{max}(h)$ has a finite connected component $C$ then \eqref{eq:violation} shows that
	$$
(f^{(d,\gamma)}\cdot h)_\mathbf{n}+v_\mathbf{n}\ge (f^{(d,\gamma)}\cdot \bar{h})_\mathbf{n}+v_\mathbf{n}=\gamma -\mathsf{N}_C(\mathbf{n})
	$$
for every $\mathbf{n}\in C$, where
	$$
\bar{h}_\mathbf{n}=
	\begin{cases}
h_\mathbf{n}&\textup{if}\enspace \mathbf{n}\in C,
	\\
0&\textup{otherwise}.
	\end{cases}
	$$
As in \eqref{eq:violation} we obtain a contradiction to \eqref{eq:RE}.
	\end{proof}

	\begin{rema}
	\label{r:sand}
Proposition \ref{p:sand} implies that $(f^{(d,\gamma )}(\sigma )(h)+\mathcal{R}_\infty ^{(\gamma )})\cap \mathcal{R}_\infty ^{(\gamma )}=\varnothing$ for every nonzero $h\in R_d$. However, if $h\in\{0,1\}^{\mathbb{Z}^d}$ satisfies that the set $\mathcal{S}(h)=\{\mathbf{n}\in\mathbb{Z}^d:h_\mathbf{n}=1\}$ is infinite and connected, then one checks easily that there exists a $v\in\mathcal{R}_\infty ^{(\gamma )}$ with $f^{(d)}(\sigma )(h)+v\in \mathcal{R}_\infty ^{(\gamma )}$. In spite of this the following result holds.
	\end{rema}

	\begin{prop}
	\label{p:injective2}
The set
	\begin{equation}
	\label{eq:injectivity}
\mathcal{V}=\{v\in \mathcal{R}_\infty ^{(\gamma )}:v+w\notin \mathcal{R_\infty ^{(\gamma )}}\enspace \textup{for every nonzero}\enspace w\in f^{(d,\gamma)}(\sigma )(\ell ^\infty (\mathbb{Z}^d,\mathbb{Z}))\}
	\end{equation}
is a dense $G_\delta $ in $\mathcal{R}_\infty ^{(\gamma )}$.
	\end{prop}

	\begin{proof}
Let $v\in \mathcal{R}_\infty ^{(\gamma )}$ and $h\in \ell ^\infty (\mathbb{Z}^d,\mathbb{Z})$ such that $\max_{\mathbf{n}\in\mathbb{Z}^d}h_\mathbf{n}\ge0$, $f^{(d,\gamma )}\cdot h\ne0$ and $v+f^{(d,\gamma )}\cdot h\in\mathcal{R}_\infty ^{(\gamma )}$. We set $M_h=\max_{\mathbf{m}\in\mathbb{Z}^d}h_\mathbf{m}$, define $\mathcal{S}_{\textup{max}}(h)\subset \mathbb{Z}^d$ as in \eqref{eq:Smax}, and put
	$$
\partial \mathcal{S}_{\textup{max}}(h)=\{\mathbf{n}\in \mathcal{S}_{\textup{max}}(h):\|\mathbf{m}-\mathbf{n}\|_\textup{max} = 1\enspace \textup{for some}\enspace \mathbf{m}\in\mathbb{Z}^d\smallsetminus \mathcal{S}_{\textup{max}}(h)\}.
	$$
As $(f^{(d,\gamma )}\cdot h)_\mathbf{n}>0$ for every $\mathbf{n}\in\partial \mathcal{S}_{\textup{max}}(h)$, the set $\partial \mathcal{S}_{\textup{max}}(h)$ must have empty intersection with
	$$
F(v)=\{\mathbf{n}\in\mathbb{Z}^d:v_\mathbf{n}=\gamma -1\}.
	$$

\smallskip Now suppose that $v\in\mathcal{R}_\infty ^{(\gamma )}$ has the following properties:
	\begin{enumerate}
	\label{form}
	\item[(a)]
The set $F(v)$ is connected;
	\item[(b)]
Every connected component of $\mathbb{Z}^d\smallsetminus F(v)$ is finite.
	\item[(c)]
$\min_{\mathbf{n}\in\mathbb{Z}^d}v_\mathbf{n}=0.$
	\end{enumerate}

According to Corollary \ref{c:sand}, every connected component $C$ of $\mathcal{S}_{\textup{max}}(h)$ is infinite. If $C\ne \mathbb{Z}^d$ then the hypotheses (a)--(b) above guarantee that the boundary $\partial C=C\cap \partial \mathcal{S}_{\textup{max}}(h)$ of $C$ is a union of finite sets, each of which is contained in one of the connected components of $\mathbb{Z}^d\smallsetminus F(v)$.

Let $C$ and $D$ be connected components of $\mathcal{S}_{\textup{max}}(h)$ and $\mathbb{Z}^d\smallsetminus F(v)$, respectively, with $D\cap\partial C\ne\varnothing $. Since $C$ is infinite and connected and $F(v)$ is connected, we must have that $h_\mathbf{m}=M_h=0$ for every $\mathbf{m}\in F(v)$.

Define $\tilde{h}$ by
	$$
\tilde{h}_\mathbf{n}=
	\begin{cases}
h_\mathbf{n}&\textup{if}\enspace \mathbf{n}\in D
	\\
0&\textup{otherwise}.
	\end{cases}
	$$
Then $(f^{(d,\gamma )}\cdot \tilde{h})_\mathbf{n}=(f^{(d,\gamma )}\cdot h)_\mathbf{n}$ for every $\mathbf{n}\in D$, and $0\le (f^{(d,\gamma )}\cdot \tilde{h})_\mathbf{n}\le (f^{(d,\gamma )}\cdot h)_\mathbf{n}$ for every $\mathbf{n}\in F(v)$. For $\mathbf{n}\in\mathbb{Z}^d\smallsetminus (F(v)\cup D)$, $(f^{(d,\gamma )}\cdot \tilde{h})_\mathbf{n}=0$. By combining these statements we see that $v+f^{(d,\gamma )}\cdot \tilde{h}\in\mathcal{R}_\infty ^{(\gamma )}$. Since $0\ne \tilde{h}\in R_d$ we obtain a contradiction to Proposition \ref{p:sand}.

This shows that $v+f^{(d,\gamma )}\cdot h\notin \mathcal{R}_\infty ^{(\gamma )}$ for every $v\in\mathcal{R}_\infty ^{(\gamma )}$ satisfying the conditions (a)--(b) above and every nonzero $h\in \ell ^\infty (\mathbb{Z}^d,\mathbb{Z})$ with $\max_{\mathbf{n}\in\mathbb{Z}^d}h_\mathbf{n}\ge0$.

If $\gamma =2d$ and $h\in\ell ^\infty (\mathbb{Z}^d,\mathbb{Z})$ satisfies that $f^{(d)}\cdot h\ne0$, then we may add a constant to $h$, if necessary, to ensure that $\max_{\mathbf{n}\in\mathbb{Z}^d}h_\mathbf{n} \ge0$. Since such an addition will not affect $f^{(d)}\cdot h$ we obtain that $v+f^{(d)}\cdot h\notin \mathcal{R}_\infty ^{(\gamma )}=\mathcal{R}_\infty $ for every $v\in\mathcal{R}_\infty $ satisfying the conditions (a)--(c) above and every nonconstant $h\in \ell ^\infty (\mathbb{Z}^d,\mathbb{Z})$.

If $\gamma >2d$ and $h\in\ell ^\infty (\mathbb{Z}^d,\mathbb{Z})$ satisfies that $\max_{\mathbf{n}\in\mathbb{Z}^d}h_\mathbf{n}<0$, then $(f^{(d,\gamma )}\cdot h)_\mathbf{n}<0$ for every $\mathbf{n}\in\mathbb{Z}^d$, and $v+f^{(d,\gamma )}\cdot h\notin \mathcal{R}_\infty ^{(\gamma )}$ for every $v\in\mathcal{R}_\infty ^{(\gamma )}$ satisfying the condition (c) above.

\smallskip Let $\mathcal{V}'\subset \mathcal{R}_\infty ^{(\gamma )}$ be the set of all points satisfying the conditions (a)--(c) \vpageref{form}. This set is clearly dense and
	\begin{equation}
	\label{eq:goodpoints}
\mathcal{V}'\subset \mathcal{V}=\{v\in\mathcal{R}_\infty ^{(\gamma )} :v+w\notin\mathcal{R}_\infty ^{(\gamma )}\enspace \textup{for every nonzero}\enspace w\in f^{(d)}(\sigma )(\ell ^\infty (\mathbb{Z}^d,\mathbb{Z}))\}.
	\end{equation}
The set $\mathcal{V}$ is therefore dense, and it is obviously shift-invariant.

In order to verify that $\mathcal{V}$ is a $G_\delta $ we write its complement as an $F_\sigma $ of the form
	\begin{align*}
\mathcal{R}_\infty ^{(\gamma )}\smallsetminus \mathcal{V}&=\smash[b]{\bigcup_{M\ge1} \bigcup_{N\ge1} \,\bigcup_{\mathbf{0}\ne \mathbf{c}\in \smash{\mathbb{Z}^{Q_M}}}}\tilde{\pi }\bigl(\{(v,h)\in \mathcal{R}_\infty ^{(\gamma )}\times B_N(\ell ^\infty (\mathbb{Z}^d,\mathbb{Z})):
	\\
&\qquad \qquad \qquad \qquad \qquad v+f^{(d)}\cdot h\in\mathcal{R}_\infty ^{(\gamma )}\enspace \textup{and}\enspace \pi _{Q_M}(f^{(d,\gamma)}\cdot h)=\mathbf{c}\}\bigr),
	\end{align*}
where $B_N(\ell ^\infty (\mathbb{Z}^d,\mathbb{Z}))=\{h\in \ell ^\infty (\mathbb{Z}^d,\mathbb{Z}):\|h\|_\infty \le N\}$, $Q_M$ appears in \eqref{eq:QM} and $\tilde{\pi }\colon \mathcal{R}_\infty ^{(\gamma )}\times \ell ^\infty (\mathbb{Z}^d,\mathbb{Z})\longrightarrow \mathcal{R}_\infty^{(\gamma)} $ is the first coordinate projection.
	\end{proof}

\section{The critical sandpile model
	\label{s:critical}
}

Throughout this section we assume that $d\ge2$ and $\gamma =2d$. We write $\mathcal{R}_\infty =\mathcal{R}_\infty ^{(2d)}$ for the critical abelian sandpile model, define the harmonic model $X_{f^{(d)}}\subset \mathbb{T}^{\mathbb{Z}^d}$ by \eqref{eq:fd} and \eqref{eq:harmonic}, and use the notation of Section \ref{s:harmonic}.

\subsection{Surjectivity of the maps $\xi _g\colon \mathcal{R}_\infty \longrightarrow X_{f^{(d)}}$} For every $g\in \tilde{I}_d$ (cf. \eqref{eq:range}) we define the map $\xi _g\colon \ell ^\infty (\mathbb{Z}^d,\mathbb{Z})\longrightarrow X_{f^{(d)}}$ by \eqref{eq:Id} and \eqref{eq:barxi}. We shall prove the following results.

	\begin{theo}
	\label{t:surjective}
For every $g\in\tilde{I}_d$, $\xi _g(\mathcal{R}_\infty )=X_{f^{(d)}}$. Furthermore, the shift-action $\sigma _{\mathcal{R}_\infty }$ of $\mathbb{Z}^d$ on $\mathcal{R}_\infty $ has topological entropy
	\begin{equation}
	\label{eq:entropy2}
	\begin{aligned}
h_\textup{top}(\sigma _{\mathcal{R}_\infty })&= \lim_{N\to\infty }\frac 1{|Q_N|}\,\log\,\bigl|\pi _{Q_N}(\mathcal{R}_\infty )\bigr|
	\\
&=\int_0^1\cdots \int_0^1 \log\,f^{(d)}(e^{2\pi it_1},\dots ,e^{2\pi it_d})\,dt_1\cdots dt_d =h(\alpha _{f^{(d)}}).
	\end{aligned}
	\end{equation}
	\end{theo}

For the proof of this result we need a bit of notation and several lemmas. For every $Q\subset \mathbb{Z}^d$ and $v\in W_d$ we set
	\begin{equation}
	\label{eq:DeltaQ}
S^{(Q)}(v)=\{v'\in W_d :\pi _{\mathbb{Z}^d\smallsetminus Q}(v')=\pi _{\mathbb{Z}^d\smallsetminus Q}(v)\},
	\end{equation}
If $V\subset W_d$ is a subset we set $S_V^{(Q)}(v)=S^{(Q)}(v)\cap V$.

We fix $g\in\tilde{I}_d$. Let $\varepsilon $ with $0<\varepsilon <1/4d$. Since $g^*\cdot w^{(d)}\in \ell ^1(\mathbb{Z}^d)$ we can find $K\ge1$ with
	\begin{equation}
	\label{eq:K}
|\bar{\xi }_g(v)_\mathbf{0}-\bar{\xi }_g(v')_\mathbf{0}|<\varepsilon \enspace \textup{for every}\enspace v,v'\in \Lambda _{2d}\enspace \textup{with}\enspace \pi _{Q_K}(v)=\pi _{Q_K}(v')
	\end{equation}
(cf. \eqref{eq:QM})

	\begin{lemm}
	\label{l:max}
Let $v\in \Lambda _{2d}$, $Q\subset \mathbb{Z}^d$ a finite set and $v'\in S_{\Lambda _{2d}}^{(Q)}(v)$ \textup{(}cf. \eqref{eq:Lambda} and \eqref{eq:DeltaQ}\textup{)}.
	\begin{enumerate}
	\item
$\xi _g(v')=\xi _g(v)$ if and only if $v'-v\in (f^{(d)})$.
	\item
If $\xi _g(v')\ne\xi _g(v)$ then
	$$
\pmb{|}\xi _g(v')_\mathbf{n}-\xi _g(v)_\mathbf{n}\pmb{|}\ge 1/4d
	$$
for some $\mathbf{n}\in Q+Q_K=\{\mathbf{m}+\mathbf{k}:\mathbf{m}\in Q,\mathbf{k}\in Q_K\}$, where $K$ is defined in \eqref{eq:K}, $Q_K$ in \eqref{eq:QM} and $\pmb{|}\cdot \pmb{|}$ in \eqref{eq:metric}.
	\end{enumerate}
	\end{lemm}

	\begin{proof}
We put $y=\bar{\xi }_g(v)$, $x=\rho (y)=\xi _g(v)$, $y'=\bar{\xi }_g(v')$ and $x'=\xi _g(v)$. Assume that
	\begin{equation}
	\label{eq:close}
\pmb{|}x'_\mathbf{n}-x_\mathbf{n}\pmb{|}<1/4d
	\end{equation}
for every $\mathbf{n}\in Q+Q_K$. Since \eqref{eq:close} holds automatically for $\mathbf{n}\in\mathbb{Z}^d\smallsetminus (Q+Q_K)$ by \eqref{eq:K}, it holds for every $\mathbf{n}\in\mathbb{Z}^d$.

We choose $z\in W_{f^{(d)}}$ with $\rho (z)=x'-x$ and $\|z_\mathbf{n}\|_\infty < 1/4d$ (cf. \eqref{eq:Wf}). Then $f^{(d)}\cdot z\in \ell ^\infty (\mathbb{Z}^d,\mathbb{Z})$, and the smallness of the coordinates of $z$ implies that $f^{(d)}\cdot z=0$.

Since $\rho (z)=\rho (y'-y)$ we obtain that $z-(y'-y)\in \ell ^\infty (\mathbb{Z}^d,\mathbb{Z})$. As the coordinates of $z$ are small and $\lim_{\mathbf{n}\to\infty }|y'-y|=\lim_{\mathbf{n}\to\infty }|\bar{\xi }_g(v'-v)|=0$ due to the continuity of $\bar{\xi }_g$, we conclude that $h=z-(y'-y)\in R_d$.

According to \eqref{eq:lemmaWf},
	$$
f^{(d)}\cdot (z-(y'-y))=f^{(d)}\cdot h=g^*\cdot (v'-v).
	$$
As $R_d$ has unique factorization and $g^*$ is not divisible by $f^{(d)}$, $v'-v$ must lie in the ideal $(f^{(d)})\subset R_d$. Theorem \ref{t:estimates} (i) and \eqref{eq:barxi} together imply that $\xi _g(v')=x'=x=\xi _g(v)$.
	\end{proof}

If $\varepsilon '>0$ and $Q\subset \mathbb{Z}^d$ we call a subset $Y\subset X_{f^{(d)}}$ \textit{$(Q,\varepsilon ')$-separated} if there exists, for every pair of distinct points $x,x'\in Y$, an $\mathbf{n}\in Q$ with $\pmb{|}x_\mathbf{n}-x'_\mathbf{n}\pmb{|}\ge\varepsilon '$. The set $Y$ is \textit{$(Q,\varepsilon ')$-spanning} if there exists, for every $x\in X_{f^{(d)}}$, an $x'\in Y$ with $\pmb{|}x_\mathbf{n}-x'_\mathbf{n}\pmb{|}<\varepsilon '$ for every $\mathbf{n}\in Q$.

	\begin{lemm}
	\label{l:spanning}
Let $Q\subset \mathbb{Z}^d$ be a finite set and $v\in \Lambda _{2d}$. Then the set $\xi _g(S_{\Lambda _{2d}}^{(Q+Q_K)}(v))$ is $(Q,\varepsilon )$-spanning.
	\end{lemm}

	\begin{proof}
According to Lemma \ref{l:surjective}, $\xi _g(\Lambda _{2d})=X_{f^{(d)}}$. If we fix $w\in \Lambda _{2d}$ and set
	$$
w'_\mathbf{n}=
	\begin{cases}
v_\mathbf{n}&\textup{if}\enspace \mathbf{n}\in Q+Q_K,
	\\
w_\mathbf{n}&\textup{otherwise},
	\end{cases}
	$$
then $w'\in S_{\Lambda _{2d}}^{(Q+Q_K)}(v)$ and $\pmb{|}\xi _g(w)_\n-\xi _g(w')_\n\pmb{|}<\varepsilon $ for every $\mathbf{n}\in Q$ by \eqref{eq:K}.
	\end{proof}

	\begin{lemm}
	\label{l:separated}
For every finite set $Q\subset \mathbb{Z}^d$ and every $w\in \mathcal{R}_\infty $, the restriction of $\xi _g$ to $S_{\mathcal{R}_\infty }^{(Q)}(w)$ is injective and the set $\xi _g(S_{\mathcal{R}_\infty }^{(Q)}(w))$ is $(Q+Q_K,1/4d)$-separated.
	\end{lemm}

	\begin{proof}
If $v,v'$ are distinct points in $S_{\mathcal{R}_\infty }^{(Q)}(w)$, then Proposition \ref{p:sand} and Lemma \ref{l:max} show that $\pmb{|}\xi _g(v)_\mathbf{n}-\xi _g(v')_\mathbf{n}\pmb{|}\linebreak[0]\ge 1/4d$ for some $\mathbf{n}\in Q+Q_K$.
	\end{proof}

We write every $h\in R_d$ as $h=\sum_{\mathbf{n}\in\mathbb{Z}^d}h_\mathbf{n}u^\mathbf{n}$ and set $\textup{supp}(h)=\{\mathbf{n}\in\mathbb{Z}^d: h_\mathbf{n}\ne0\}$. For $Q\subset \mathbb{Z}^d$ we put
	\begin{equation}
	\label{eq:R+}
	\begin{gathered}
R(Q)=\{h\in R_d: \textup{supp}(h)\subset Q\},
	\\
R^+(Q)=\{h\in R(Q): h_\mathbf{n}\ge0\enspace \textup{for every}\enspace \mathbf{n}\in\mathbb{Z}^d\},
	\\
S^+(Q)=\{h\in R(Q):h_\mathbf{n}\in\{0,1\}\enspace \textup{for every}\enspace \mathbf{n}\in\mathbb{Z}^d\}.
	\end{gathered}
	\end{equation}
For $L\ge1$, $v\in \Lambda _{2d}$ and $q\ge0$ we set
	\begin{align}
Y_v(q)&=\{w\in S ^{(Q_{L+K+1})}(v):\textup{for every}\enspace \mathbf{n}\in\mathbb{Z}^d, \, 0\le w_\mathbf{n}<2d\enspace \textup{if}\enspace \|\mathbf{n}\|_\textup{max}\ne L+K+1\notag
	\\
&\qquad \qquad \qquad \qquad \qquad \textup{and}\enspace -q\le w_\mathbf{n}<2d\enspace \textup{if}\enspace \|\mathbf{n}\|_\textup{max}=L+K+1\},
	\label{eq:Y}
	\\
&\qquad \qquad Y_v'(q)=\{w\in Y_v(q):\pi _{Q_{L+K}}(w)\in \pi _{Q_{L+K}}(\mathcal{R}_\infty )\notag\}.
	\end{align}

	\begin{lemm}
	\label{l:cover}
Let $L\ge1$, $q\ge0$ and $v\in \Lambda _{2d}$. Then
	\begin{equation}
	\label{eq:equality}
Y_v'(q)= Y_v(q)\smallsetminus \hspace{-4mm}\bigcup_{0\ne h\in S^+(Q_{L+K})}\hspace{-4mm}(Y_v(q+1)-h\cdot f^{(d)}).
	\end{equation}
	\end{lemm}

	\begin{proof}
Suppose that $v\in Y_v'(q)$. According to the proof of Proposition \ref{p:sand} there exists, for every nonzero $h\in S^+(Q_{L+K})$, an $\mathbf{n}\in \textup{supp}(h)\subset Q_{L+K}$ with $(v+h\cdot f^{(d)})_\mathbf{n}> 2d-1$. In particular, $v+h\cdot f^{(d)}(q)\notin Y_v(q+1)$ and $v\notin Y_v(q+1)-h\cdot f^{(d)}$. This shows that
	$$
Y_v'(q)\subset Y_v(q)\smallsetminus \hspace{-4mm}\bigcup_{0\ne h\in S^+(Q_{L+K})}\hspace{-4mm}(Y_v(q+1)-h\cdot f^{(d)}).
	$$

Conversely, if $v\in Y_v(q)\smallsetminus \bigcup_{0\ne h\in S^+(Q_{L+K})} (Y_v(q+1)-h\cdot f^{(d)})$, but $v\notin Y_v'(q)$, then the proof of Proposition \ref{p:sand} allows us to find a nonzero $h\in S^+(Q_{L+K})$ with $(v+h\cdot f^{(d)})_\mathbf{n}<2d$ for every $\mathbf{n}\in \textup{supp}(h)$. If $(v+h\cdot f^{(d)})_\mathbf{n}<0$ for some $\mathbf{n}\in Q_{L+K}$, then $\mathbf{n}\notin\textup{supp}(h)$ and $-2d\le (v+h\cdot f^{(d)})_\mathbf{n}<0$. We replace $h$ by $h'=h+u^\mathbf{n}\in S^+(Q_{L+K})$ and obtain that $0\le (v+h'\cdot f^{(d)})_\mathbf{n}<2d$ for every $\mathbf{n}\in \textup{supp}(h')$. By repeating this process we can find $h''\in S^+(Q_{L+K})$ with $\textup{supp}(h'')\supset \textup{supp}(h)$ such that $0\le (v+h''\cdot f^{(d)})_\mathbf{n}\le 2d-1$ for every $\mathbf{n}\in Q_{L+K}$. Since $0\ge (h''\cdot f^{(d)})_\mathbf{n}\ge -1$ if $\|\mathbf{n}\|_\textup{max}=L+K+1$ and $(h''\cdot f^{(d)})_\mathbf{n}=0$ outside $Q_{L+K+1}$ we see that $v+h''\cdot f^{(d)}\in Y_v'(q+1)$. This contradicts our choice of $v$ and proves \eqref{eq:equality}.
	\end{proof}

	\begin{lemm}
	\label{l:section}
For every $v\in\Lambda _{2d}$ and $L\ge1$ there exists an $h\in R^+(Q_L)$ with $v'=v+h\cdot f^{(d)}\in Y_v'((2d-1)\cdot (2L+1)^d)$.
	\end{lemm}

	\begin{proof}
For every $v\in \ell ^\infty (\mathbb{Z}^d,\mathbb{Z})$ we define $D_{Q_{L+1}}(v)$ by \eqref{eq:D}. Since $D_{Q_{L+1}}(v+u^\mathbf{n}\cdot f^{(d)})\le D_{Q_{L+1}}(v)-2$ for every $\mathbf{n}\in Q_L$, $D_{Q_{L+1}}(v+h\cdot f^{(d)})\le D_{Q_{L+1}}(v)-2\|h\|_1$ for every $h\in S^+(Q_L)$.

Suppose that $v\in \Lambda _{2d}$. If $w\notin Y_v'(0)$ then \eqref{eq:equality} shows that we can find a nonzero $h^{(1)}\in S^+(Q_L)$ with $v^{(1)}=v+h^{(1)}\cdot f^{(d)}\in Y_v(1)$, and the first paragraph of this proof shows that $D_{Q_{L+1}}(v^{(1)})\le D_{Q_{L+1}}(v)-2\|h^{(1)}\|_1$.

If $v^{(1)}\notin Y_v'(1)$ we can repeat this argument and find a nonzero $h^{(2)}\in S^+(Q_L)$ with $v^{(2)}=v^{(1)}+h^{(2)}\cdot f^{(d)}\in Y_v(2)$ and $D_{Q_{L+1}}(v^{(2)})\le D_{Q_{L+1}}(v)-2\|h^{(1)}\|_1-2\|h^{(2)}\|_1$.

Proceeding by induction, we choose nonzero elements $h^{(1)},\dots ,h^{(m)}\in S^+(Q_L)$ with $v^{(k)}=v+(h^{(1)}+\dots +h^{(k)})\cdot f^{(d)}\in Y_v(m)$ for every $k=1,\dots ,m$.

We claim that $v^{(k)}\in Y_v((2d-1)\cdot (2L+1)^d)$ for every $k\ge1$, and that this process has to stop, i.e., that
	\begin{equation}
	\label{eq:claim}
v'=v^{(m)}=v+(h^{(1)}+\dots +h^{(m)})\cdot f^{(d)}\in Y_v'((2d-1)\cdot (2L+1)^d)
	\end{equation}
for some $m\ge1$.

In order to verify this we assume that we have found $h^{(1)},\dots ,h^{(k)}\linebreak[0]\in S^+(L)$ with $v^{(k)}=v+(h^{(1)}+\dots +h^{(k)})\cdot f^{(d)}\in Y_v(k)$. Since $\sum_{\mathbf{n}\in Q_{L+1}} v^{(k)}_\mathbf{n}=\sum_{\mathbf{n}\in Q_{L+1}} v_\mathbf{n}$, $0\le v^{(k)}_\mathbf{n}\le2d-1$ for $\mathbf{n}\in Q_L$, $v^{(k)}_\mathbf{n}\le v_\mathbf{n}$ if $\|\mathbf{n}\|_\textup{max}=L+1$ and $v^{(k)}_\mathbf{n}=v_\mathbf{n}$ for every $\mathbf{n}\notin Q_{L+1}$, we know that
	\begin{equation}
	\label{eq:v'}
	\begin{aligned}
(2d-1)\cdot 2d\cdot (2L+1)^{d-1}&\ge \sum_{\{\mathbf{n}:\|\mathbf{n}\|_\textup{max}=L+1\}} v_\mathbf{n}\ge \sum_{\{\mathbf{n}:\|\mathbf{n}\|_\textup{max}=L+1\}} v^{(k)}_\mathbf{n}
	\\
&\ge \sum_{\{\mathbf{n}:\|\mathbf{n}\|_\textup{max}=L+1\}}v_\mathbf{n} -\sum_{\mathbf{n}\in Q_L}v^{(k)}_\mathbf{n}
	\\
&\ge -(2d-1)\cdot (2L+1)^d,
	\end{aligned}
	\end{equation}
so that $v^{(k)}\in Y_v((2d-1)\cdot (2L+1)^d)$ for every $k\ge1$.

Furthermore,
	\begin{align*}
D_{Q_{L+1}}(v^{(k)})&=\smash[t]{D_{Q_{L+1}}(v)-2\sum_{j=1}^k \|h^{(j)}\|_1\le D_{Q_{L+1}}(v)-2k}
	\\
&< (L+1)^2\cdot (2d-1)\cdot (2L+3)^d -2k
	\end{align*}
and
	$$
D_{Q_{L+1}}(v^{(k)})\ge -(L+1)^2\cdot (2d-1)\cdot (2L+1)^d\cdot |Q_{L+1}\smallsetminus Q_L|
	$$
for every $k$, so that the integer $k$ has to remain bounded. This shows that our inductive process has to terminate, which proves \eqref{eq:claim}.
	\end{proof}

Before we complete the proof of Theorem \ref{t:surjective} we state another consequence of the Lemmas \ref{l:cover} and \ref{l:section}.

	\begin{prop}
	\label{p:correction}
Let $v\in\ell ^\infty (\mathbb{Z}^d,\mathbb{Z})$ and $M\ge1$. Then there exists a unique $h\in R_d$ with the following properties.
	\begin{enumerate}
	\item
$\textup{supp}(h)=\{\mathbf{m}\in\mathbb{Z}^d: h_\mathbf{m}\ne0\} \subset Q_M$;
	\item
If $v'=v+h\cdot f^{(d)}$, then $\pi _{Q_M}(v')\in \pi _{Q_M}(\mathcal{R}_\infty )$;
	\item
$v_\mathbf{m}=v'_\mathbf{m}$ for every $\mathbf{m}\in\mathbb{Z}^d$ with $\|\mathbf{m}\|_\textup{max}>M+1$;
	\item
$\sum_{\{\mathbf{n}: \|\mathbf{n}\|_\textup{max}=M+1\}} |v'_\mathbf{n}|\le (2M+3)^d\cdot \|v\|_\infty $.
	\end{enumerate}
	\end{prop}

	\begin{proof}
The proof of Lemma \ref{l:cover} allows us to find a polynomial $h^-\in R_d$ with nonnegative coefficients and $\textup{supp}(h^-)\subset Q_M$ such that $(v-h^-\cdot f^{(d)})_\mathbf{n}<2d$ for every $\mathbf{n}\in Q_M$. Next we proceed as in the proof of Lemma \ref{l:section} and choose a polynomial $h^+\in R_d$ with nonnegative coefficients and $\textup{supp}(h^+)\subset Q_M$ such that $v'=v+(h^+-h^-)\cdot f^{(d)}$ satisfies (2). Condition (3) holds obviously, and (4) follows from the fact that $\sum_{\mathbf{n}\in Q_{M+1}}v_\mathbf{n}=\sum_{\mathbf{n}\in Q_{M+1}}v_\mathbf{n}'$.

In order to verify the uniqueness of $h=h^+-h^-$ we assume that $h'\in R_d$ is another polynomial with $\textup{supp}(h')\subset Q_M$ such that $v''=v+h'\cdot f^{(d)}$ satisfies Condition (2) above. We assume without loss in generality that $h_\mathbf{m}>h'_\mathbf{m}$ for some $\mathbf{m}\in Q_M$ and set $g=h-h'$ and
	$$
w_\mathbf{n}=
	\begin{cases}
v_\mathbf{n}''&\textup{if}\enspace \mathbf{n}\in Q_M,
	\\
2d&\textup{otherwise}.
	\end{cases}
	$$
Then $w\in\mathcal{R}_\infty $ and $(w+g\cdot g^{(d)})_\mathbf{n}=v_\mathbf{n}<2d$ for every $\mathbf{n}\in Q_M$. Since $\textup{supp}(g)\subset Q_M$ and $g_\mathbf{n}>0$ for some $\mathbf{n}\in Q_M$ this contradicts Proposition \ref{p:sand}.
	\end{proof}

	\begin{proof}[Proof of Theorem \ref{t:surjective}]
We fix $\varepsilon >0$ and choose $K$ according to \eqref{eq:K}. Lemma \ref{l:section} and \eqref{eq:v'} show that $X_{f^{(d)}}=\xi _g(\Lambda _{2d})=\xi _g(\Lambda _{2d}(L+K+1,(2d-1)\cdot (2L+2K+1)^d))$, where
	\begin{equation}
	\label{eq:LambdaM}
	\begin{aligned}
\Lambda _{2d}(M,q)&=\Bigl\{v\in\ell ^\infty (\mathbb{Z}^d,\mathbb{Z}):v_\mathbf{m}<2d\enspace \textup{for every}\enspace \mathbf{n}\in\mathbb{Z}^d,
	\\
&\qquad v_\mathbf{n}\ge0\enspace \textup{for every}\enspace \mathbf{n}\in\mathbb{Z}^d\enspace \textup{with}\enspace \|\mathbf{n}\|_\textup{max} >M+1,
	\\
&\qquad \sum\nolimits_{\{\mathbf{n}\in\mathbb{Z}^d: \|\mathbf{n}\|_\textup{max}=M+1\}}v_\mathbf{n}\ge -q\enspace \textup{and}\enspace \pi _{Q_M}(v)\in\pi _{Q_M}(\mathcal{R}_\infty )\Bigr\}.
	\end{aligned}
	\end{equation}
Exactly the same argument as in the proof of Lemma \ref{l:surjective} shows that $\xi _g(\mathcal{R}_\infty )=X_{f^{(d)}}$.

Since $\xi _g(\mathcal{R}_\infty )=X_{f^{(d)}}$ we know that
	\begin{equation}
	\label{eq:entropy3}
h_\textup{top}(\sigma _{\mathcal{R}_\infty })\ge h_\textup{top}(\alpha _{f^{(d)}})=\int_0^1\cdots \int_0^1 \log\,f^{(d)}(e^{2\pi is_1},\dots ,e^{2\pi is_d})\,ds_1\cdots ds_d
	\end{equation}
(cf. \cite{LSW} or \cite[Theorem 18.1]{DSAO}).

In order to prove the reverse inequality we note that $\xi _g$ is injective on $S_{\mathcal{R}_\infty }^{(Q_L)}(v)$ for every $v\in\mathcal{R}_\infty $ and $L\ge1$ and that $\xi _g(S_{\mathcal{R}_\infty }^{(Q_L)}(v))$ is a $(Q_{L+K},1/4d)$-separated subset of $X_{f^{(d)}}$, by Proposition \ref{p:sand} and Lemma \ref{l:max}. In particular, if $\bar{v}\in\mathcal{R}_\infty $ is given by
	\begin{equation}
	\label{eq:barv}
\bar{v}_\mathbf{n}=2d-1\enspace \textup{for every}\enspace \mathbf{n}\in\mathbb{Z}^d,
	\end{equation}
then $\bigl|\pi _{Q_L}(S_{\mathcal{R}_\infty }^{(Q_L)}(\bar{v}))\bigr|=\bigl|\pi _{Q_L}(\mathcal{R}_\infty )\bigr|$ for every $L\ge1$.

For every $L\ge0$ we denote by $n(L+K)$ the maximal size of a $(Q_{L+K},1/4d)$-separated set in $X_{f^{(d)}}$. From the definition of topological entropy we obtain that
	\begin{equation}
	\label{eq:Deltasize}
	\begin{aligned}
h_\textup{top}(\sigma _{\mathcal{R}_\infty })&=\lim_{L\to\infty }\frac{1}{|Q_L|} \log\,\bigl|\pi _{Q_L}(\mathcal{R}_\infty )\bigr|=\lim_{L\to\infty }\frac{1}{|Q_L|} \log\,\bigl|S_{\mathcal{R}_\infty }^{(Q_L)}(\bar{v})\bigr|
	\\
&=\lim_{L\to\infty }\frac{1}{|Q_L|} \log\,\bigl|\xi _g(S_{\mathcal{R}_\infty }^{(Q_L)}(\bar{v}))\bigr|\le \lim_{L\to\infty }\frac{1}{|Q_L|} \log n(L+K)
	\\
& = \lim_{L\to\infty }\frac{1}{|Q_{L+K}|} \log n(L+K) = h_\textup{top}(\alpha _f^{(d)}),
	\end{aligned}
	\end{equation}
which completes the proof of the theorem.
	\end{proof}

	\begin{rema}
	\label{r:surjective}
The expression \eqref{eq:entropy} for the topological entropy of $\sigma _{\mathcal{R}_\infty }$ can be found in \cite[p. 56]{Dhar2}. By using the fact that $\alpha _{f^{(d)}}$ and $\sigma _{\mathcal{R}_\infty }$ have the same topological entropy one can prove Theorem \ref{t:surjective} a little more directly: the Lemmas \ref{l:spanning} and \ref{l:separated} imply that the restriction of $\alpha $ to the closed, shift-invariant subset $\xi _g(\mathcal{R}_\infty )\subset X_{f^{(d)}}$ has the same topological entropy as $\alpha _{f^{(d)}}$. Since the Haar measure $\lambda _{X_{f^{(d)}}}$ is the unique measure of maximal entropy for $\alpha _{f^{(d)}}$ by \cite{LSW}, $\xi _g(\mathcal{R}_\infty )$ has to coincide with $X_{f^{(d)}}$, as claimed in Theorem \ref{t:surjective}.
	\end{rema}

	\begin{theo}
	\label{t:sand2}
For every $w\in \mathcal{R}_\infty $ and $L\ge1$ we denote by $\nu _L^{(w)}$ the equidistributed probability measure on the set $S_{\mathcal{R}_\infty }^{(Q_L)}(w)$ in \eqref{eq:DeltaQ}. Fix $w\in \mathcal{R}_\infty $ and let $\mu ^{(w)}$ be any limit point of the sequence of probability measures
	$$
\mu _L^{(w)}=\frac1{|Q_L|} \sum_{\mathbf{k}\in Q_L}\sigma _*^\mathbf{k}\nu _L^{(w)}
	$$
as $L\to\infty $. Then $\mu ^{(w)}$ is a measure of maximal entropy on $\mathcal{R}_\infty $ and $(\xi _g)_*\mu ^{(w)}=\lambda _{X_{f^{(d)}}}$ for every $g\in \tilde{I}_d$.

In fact, if $\mu $ is any shift-invariant probability measure of maximal entropy on $\mathcal{R}_\infty $, then $(\xi _g)_*\mu =\lambda _{X_{f^{(d)}}}$ for every $g\in \tilde{I}_d$.
	\end{theo}

	\begin{proof}
We fix $w\in\mathcal{R}_\infty $. Let $L\ge1$ and let $\tilde{\nu }_L=(\xi _g)_*\nu _L^{(w)}$ be the equidistributed probability measure on the $(Q_{L+K},1/4d)$-separated set $\xi _g(S_{\mathcal{R}_\infty }^{(Q_L)}(w))$ of cardinality $\ge \bigl|\pi _{Q_{L-1}}(\mathcal{R}_\infty )\bigr|$.

We set $\tilde{\mu }_L^{(w)}=(\xi _g)_*\mu _L^{(w)}=\frac1{|Q_L|} \sum_{\mathbf{k}\in Q_L}(\alpha _{f^{(d)}}^\mathbf{k})_*\tilde{\nu }_L^{(w)}$. By choosing a suitable subsequence $(L_k,\,k\ge1)$ of the natural numbers we may assume that $\lim_{k\to\infty }\mu _{L_k}^{(w)}=\mu ^{(w)}$ and $\lim_{k\to\infty }\tilde{\mu }_{L_k}^{(w)}\linebreak[0]=\tilde{\mu }^{(w)}=(\xi _g)_*\mu ^{(w)}$.

We denote by $\mu '=(\pi _{\{\mathbf{0}\}})_*\tilde{\mu }^{(w)}$ the projection of $\tilde{\mu }^{(w)}$ onto the zero coordinate in $X_{f^{(d)}}$ and choose a partition $\{I_1,\dots ,I_{8d}\}$ of $\mathbb{T}$ into half-open intervals of length $1/8d$ such that the endpoints of these intervals all have $\mu '$-measure zero. For $i=1,\dots ,8d$ we set $A_i=\{x\in X_{f^{(d)}}:x_\mathbf{0}\in I_i\}$ and observe that $\tilde{\mu }^{(w)}(\partial A_i)=0$. We write $\zeta =\{A_1,\dots ,A_{8d}\}$ for the resulting partition of $X_{f^{(d)}}$.

For every $L\ge1$ we set $\zeta _L=\bigvee_{\mathbf{k}\in Q_{L+K}}\alpha _{f^{(d)}}^{-\mathbf{k}}(\zeta )$. Since each atom of $\zeta _L$ contains at most one atom of $\tilde{\nu }_L^{(w)}$ (by Lemma \ref{l:separated}) and all these atoms have equal mass, $H_{\tilde{\nu }_L^{(w)}}(\zeta _L)=\log\,|S_{\mathcal{R}_\infty }^{(Q_L)}(w)|$.

Exactly the same argument as in the proof of the inequality $(*)$ in \cite[Theorem 8.6]{Walters} shows that, for every $M,L\ge1$ with $2M+2K<L$,
	\begin{align*}
\frac{|Q_M|}{|Q_L|}\log\,|S_{\mathcal{R}_\infty }^{(Q_L)}(w)|&=H_{\tilde{\nu }_L^{(w)}}(\zeta _M)\le H_{\tilde{\mu }_L^{(w)}}(\zeta _M)
	\\
&+\frac{|Q_{M+K}|\cdot (|Q_{L+K}|-|Q_{L-M-K}|}{|Q_L|}\cdot \log (8d).
	\end{align*}
By setting $L=L_k$ and letting $k\to\infty $ we obtain from \eqref{eq:Deltasize} that
	$$
|Q_M|\cdot h_\textup{top}(\alpha _{f^{(d)}})\le \lim_{k\to\infty }H_{\tilde{\mu }_{L_k}^{(w)}}(\zeta _M) = H_{\tilde{\mu }^{(w)}}(\zeta _M)
	$$
for every $M\ge1$, and hence that
	$$
h_\textup{top}(\alpha _{f^{(d)}})\le \lim_{M\to\infty }\frac{1}{|Q_{M+K}|}\cdot H_{\tilde{\mu }^{(w)}}(\zeta _{M})=h_{\tilde{\mu }^{(w)}}(\alpha _{f^{(d)}}).
	$$
Since $\lambda _{X_{f^{(d)}}}$ is the unique measure of maximal entropy on $X_{f^{(d)}}$, $\tilde{\mu }^{(w)}$ coincides with $\lambda _{X_{f^{(d)}}}$, and $\mu ^{(w)}$ is a measure of maximal entropy on $\mathcal{R}_\infty $.

In order to complete the proof of Theorem \ref{t:sand2} we assume that $\mu $ is an arbitrary ergodic shift-invariant probability measure with maximal entropy on $\mathcal{R}_\infty ^{(\gamma )}$. We let $M\ge5$, put $F=\pi _{Q_M}(\mathcal{R}_\infty )$ and set, for every $z\in F$,
	$$
\mathcal{O}_z=\{v\in \mathcal{R}_\infty :\pi _{Q_M}(v)=z\}.
	$$
Fix $z\in F$ with $c=\mu (\mathcal{O}_z)>0$. The ergodic theorem guarantees that
	\begin{equation}
	\label{eq:average}
\lim_{N\to\infty }\frac{1}{|Q_N|}\sum_{\mathbf{m}\in Q_N}1_\mathcal{O}(\sigma ^{3M\mathbf{m}}v)=c
	\end{equation}
for $\mu \emph{-a.e.}\,v\in\mathcal{R}_\infty $. Let $z'\in F$ be given by
	$$
z'_\mathbf{n}=
	\begin{cases}
2d-1&\textup{if}\enspace \|\mathbf{n}\|_\textup{max}=M,
	\\
z_\mathbf{n}&\textup{if}\enspace \mathbf{n}\in Q_{M-1}.
	\end{cases}
	$$

We claim that $\mu (\mathcal{O}_{z'})>0$. In order to see this we assume that $\mu (\mathcal{O}_{z'})=0$ (which implies, of course, that $z\ne z'$). If $v\in\mathcal{R}_\infty $ is fixed for the moment, and if $S_v=\{\mathbf{n}\in\mathbb{Z}^d:\sigma ^{3M\mathbf{n}}v\in\mathcal{O}_z\}$, then we can replace the coordinates of $\sigma ^{3M\mathbf{m}}v$ in $Q_M$ by those of $z'$ for every $\mathbf{m}\in S_v$, and we can do so independently at every $\mathbf{m}\in S_v$. The resulting points $v'$ will always lie in $\mathcal{R}_\infty $. An elementary entropy argument shows that we could increase the entropy of $\mu $ under the $\mathbb{Z}^d$-action $\mathbf{n}\rightarrow \sigma ^{3M\mathbf{n}}$ by making all these points $v'$ equally likely, which would violate the maximality of the entropy of $\mu $ (a more formal argument should be given in terms of conditional measures).

Exactly the same kind of argument as in the preceding paragraph allows us to conclude that the cylinder sets $\mathcal{O}_{z''}$ with $z_\mathbf{n}''\in F$ and
	$$
z_\mathbf{n}''=2d-1\enspace \textup{for every}\enspace \mathbf{n}\in Q_M\enspace \textup{with}\enspace\|\mathbf{n}\|_\textup{max}=M,
	$$
all have equal measure. A slight modification of the proof of the first part of this theorem now shows that $h((\xi _g)_*\mu )=h(\lambda _{X_{f^{(d)}}})$, i.e., that $(\xi _g)_*\mu =\lambda _{X_{f^{(d)}}}$.
	\end{proof}

\subsection{Properties of the maps $\xi _g,\,g\in \tilde{I}_d$
	\label{ss:properties}
}

\subsubsection{\label{ss:group}The `group structure' of $\mathcal{R}_\infty $} In \eqref{eq:entropy} we saw that $\sigma _{\mathcal{R}_\infty }$ and $\alpha _{f^{(d)}}$ have the same topological entropy. If $\mu $ is a shift-invariant measure of maximal entropy on $\mathcal{R}_\infty $, then the dynamical system $(\mathcal{R}_\infty ,\mu ,\sigma _{\mathcal{R}_\infty })$ has a Bernoulli factor of full entropy (cf. \cite{Sinai}). As $(X_{f^{(d)}},\lambda _{X_{f^{(d)}}},\alpha _{f^{(d)}})$ is Bernoulli by \cite{RS}, the full entropy Bernoulli factor of $(\mathcal{R}_\infty ,\mu ,\sigma _{\mathcal{R}_\infty })$ is measurably conjugate to $(X_{f^{(d)}},\lambda _{X_{f^{(d)}}},\alpha _{f^{(d)}})$. In particular, there exists a $\mu $-\textit{a.e.} defined measurable map $\phi \colon \mathcal{R}_\infty \longrightarrow X_{f^{(d)}}$ with $\phi _*\mu =\lambda _{X_{f^{(d)}}}$ and $\phi \circ \sigma _{\mathcal{R}_\infty }=\alpha _{f^{(d)}}\circ \phi $ $\mu $-\textit{a.e.}

What distinguishes the maps $\xi _g,\,g\in \tilde{I}_d$, from these abstract factor maps $\phi \colon \mathcal{R}_\infty \longrightarrow X_{f^{(d)}}$ is that the $\xi _g$ are not only continuous and surjective, but that they also reflect the somewhat elusive group structure of $\mathcal{R}_\infty $ in the following sense.

It is well known that the set $\mathcal{R}_E$ of recurrent sandpile configurations on a finite set $E\subset \mathbb{Z}^d$ in \eqref{eq:RE} is a group (cf. \cite{Dhar0}, \cite{Dhar1}, \cite{Dhar2}). However, the group operation does not extend in any immediate way to the infinite sandpile model $\mathcal{R}_\infty $.

Fix $g\in \tilde{I}_d$ and suppose that $v,v'\in\mathcal{R}_\infty $, and that $w=v+v'\in\Lambda _{4d-1}$ (with coordinate-wise addition). Proposition \ref{p:correction} shows that there exists, for every $M\ge1$, an element $w^{(M)}\in\ell ^\infty (\mathbb{Z}^d,\mathbb{Z})$ satisfying the conditions (1)--(4) there. Since $w-w^{(M)}\in (f^{(d)})$ for every $M\ge1$, $\xi _g(w^{(M)})=\xi _g(w)$ for every $M\ge1$. Exactly as in the proof of Lemma we observe that any coordinate-wise limit $\tilde{w}\in\mathcal{R}_\infty $ of the sequence $(w^{(M)},\,M\ge1)$ still satisfies that $\xi _g(\tilde{w})=\xi _g(w)=\xi _g(v)+\xi _g(v')$.

The `sum' $\tilde{w}$ of $v$ and $v'$ is, of course, not uniquely defined, but any two versions of this sum are identified under $\xi _g$.

Moreover, if $\sim $ is the equivalence relation on $\mathcal{R}_\infty $ defined by $v\sim v'$ if and only if $v-v'\in\ker(\xi _{I_d})=K_d$ (cf. \eqref{eq:kerI}), then $\mathcal{R}_\infty /_\sim $ is a compact abelian group isomorphic to $\tilde{X}_{f^{(d)}}=X_{f^{(d)}}/X_{I_d}$ (cf. Lemma \ref{l:xiI}): if $[v]$ is the equivalence class of $v\in\mathcal{R}_\infty $, then the map $\theta _d\circ \xi _{I_d}\colon \mathcal{R}_\infty \longrightarrow \tilde{X}_{f^{(d)}}$ in \eqref{eq:theta} sends $[v]$ to $\theta _d\circ \xi _{I_d}(v)$ and maps the group operation $[v]\oplus [v']\defeq[v+v']$ on $\mathcal{R}_\infty /_\sim $ to that on $\tilde{X}_{f^{(d)}}$.

\subsubsection{\label{ss:injectivity}The problem of injectivity}

In Subsection \ref{ss:group} we saw that $\mathcal{R}_\infty $ has a natural group structure modulo elements in the kernel of $\xi _g$. Another problem which depends on the intersection of $\mathcal{R}_d$ with the cosets of $\ker \xi _{I_d}$ is the question of `pulling back' to $\mathcal{R}_\infty $ dynamical properties of $\alpha _{f^{(d)}}$, such as uniqueness or the Bernoulli property of the measure of maximal entropy of $\mathcal{R}_\infty $.

It is clear that the map $\xi _{I_d}$ (and hence all the maps $\xi _g,\,g\in \tilde{I}_d$) must be noninjective on $\mathcal{R}_\infty $, since these maps are continuous, $\mathcal{R}_d$ is zero-dimensional, and the groups $X_{f^{(d)}}$ and $\bar{X}_{f^{(d)}}$ are connected. The following lemma shows that some of the maps $\xi _g,\,g\in \tilde{I}_d$, are `more injective' than others and is the reason for determining the ideal $I_d$ precisely in Section \ref{s:green}.

	\begin{lemm}
	\label{l:noninjective}
Let $g\in\tilde{I}_d$ and $h\in R_d$. For every $v,w\in \mathcal{R}_\infty $ with $\xi _g(w)\in\xi _g(v)+\ker h(\alpha )$, $\xi _{g\cdot h}(v)=\xi _{g\cdot h}(w)$. It follows that
	\begin{equation}
	\label{eq:kernel}
|\{w\in \mathcal{R}_\infty:\xi _{g\cdot h}(w)=\xi _{g\cdot h}(v)\}| = |\ker h(\alpha _{f^{(d)}})|
	\end{equation}
for every $v\in\mathcal{R}_\infty $.
	\end{lemm}

	\begin{proof}
If $x=\xi _g(v)$, $y\in\ker h(\alpha _{f^{(d)}})$ and $w\in\mathcal{R}_\infty $ satisfies that $\xi _g(w)=x+y$ (cf. Theorem \ref{t:surjective}), then
	\begin{equation}
\xi _{g\cdot h}(w)=h(\alpha )(\xi _g(w))=h(\alpha )(x+y)= h(\alpha )(x)=\xi _{g\cdot h}(v).\tag*{$\square$}
	\end{equation}
\renewcommand{\qedsymbol}{}\vspace{-5mm}
	\end{proof}

\section{The dissipative sandpile model
	\label{s:dissipative}
}

In this section we fix $d\ge2$ and $\gamma >2d$, and consider the dissipative sandpile model $\mathcal{R}_\infty ^\gamma \subset \Lambda _{\gamma }$ described in Section \ref{s:sandpiles} and investigated in \cite{dis1,dis2,MRS}.

\subsection{The dissipative harmonic model} Consider the Laurent polynomial $f^{(d,\gamma )}\in R_d$ defined in \eqref{eq:fdgamma} and the corresponding compact abelian group
	\begin{equation}
	\label{eq:Xfdgamma}
	\begin{aligned}
X_{f^{(d,\gamma )}}= \ker f^{(d,\gamma )}(\alpha ) &= \biggl\{x=(x_\mathbf{n})_{\mathbf{n}\in\mathbb{Z}^d}\in \mathbb{T}^{\mathbb{Z}^d}: \gamma x_\mathbf{n} -\sum_{i=1}^d (x_{\mathbf{n}+\mathbf{e}^{(i)}}+ x_{\mathbf{n}-\mathbf{e}^{(i)}})=0
	\\
&\qquad \qquad \qquad \qquad \qquad \qquad \qquad \qquad \enspace \enspace \textup{for every}\enspace \mathbf{n}\in\mathbb{Z}^d\smash{\biggr\}}.
	\end{aligned}
	\end{equation}
We write $\alpha _{X_{f^{(d,\gamma )}}}$ for the shift-action \eqref{eq:alpha} of $\mathbb{Z}^d$ on $X_{f^{(d,\gamma )}}\subset \mathbb{T}^{\mathbb{Z}^d}$.

	\begin{lemm}
The shift-action $\alpha $ of $\mathbb{Z}^d$ on $X_{f^{(d,\gamma )}}$ is expansive, i.e., there exists an $\epsilon >0$ such that
	$$
\sup_{\mathbf{n}\in\mathbb{Z}^d}\pmb{|}x_\mathbf{n}-x_\mathbf{n}'\pmb{|}>\epsilon
	$$
for every $x,x'\in X_{f^{(d,\gamma )}}$ with $x\ne x'$.

The entropy of $\alpha _{f^{(d,\gamma )}}$ is given by
	$$
h_\textup{top}(\alpha _{f^{(d,\gamma )}})=h_{\lambda _{X_{f^{(d,\gamma )}}}}(\alpha _{f^{(d,\gamma )}}) =\int_{0}^1\cdots\int_{0}^1 \log f^{(d,\gamma )}(e^{2\pi it_1},\dots ,e^{2\pi it_d}) \,dt_1\cdots dt_d,
	$$
and the Haar measure $\lambda_{X_{f^{(d,\gamma )}}}$ is the unique shift-invariant measure of maximal entropy on $X_{f^{(d,\gamma )}}$.
	\end{lemm}

	\begin{proof}
Since $f^{(d,\gamma )}$ has no zeros in
	$$
\mathbb S^d=\bigl\{(z_1,\dots,z_d)\in \mathbb C^d: |z_i|=1 \enspace \textup{for} \enspace i=1,\ldots,d\bigr\},
	$$
$\alpha _{f^{(d,\gamma )}}$ is expansive by \cite[Theorem 6.5]{DSAO}. The last two statements follow from \cite[Theorems 19.5, 20.8 and 20.15]{DSAO}.
	\end{proof}

\subsection{The covering map $\xi ^{(\gamma )}\colon \mathcal{R}_\infty ^{(\gamma )}\longrightarrow X_{f^{(d,\gamma )}}$} Since $\alpha _{f^{(d,\gamma )}}$ is expansive and has completely positive entropy, the equation
	\begin{equation}
	\label{GreenGamma}
f ^{(d,\gamma )}\cdot w=1
	\end{equation}
has a unique solution $w=w^{(d,\gamma )}\in\ell ^1(\mathbb{Z}^d)$, given by
	$$
w^{(d,\gamma )}_\mathbf{n} = \int_{0}^1\cdots\int_{0}^1 \frac {e^{-2\pi i\langle \mathbf{n},\mathbf{t}\rangle}} {\gamma -2\cdot \sum_{i=1}^d\cos(2\pi t_i)} \,dt_1\cdots dt_d,
	$$
where $\mathbf{t}=(t_1,\dots ,t_d)$ (cf. \eqref{eq:inverse}, \cite{LS} and \cite{deBHR}). Since $w^{(d,\gamma )}\in\ell ^1(\mathbb{Z}^d)$, we can proceed as in \eqref{eq:barxi} and define a homomorphism $\xi ^{(\gamma )}: \mathcal{R}^{(\gamma )}_\infty\to X_{f^{(\gamma ,d)}}$ by
	$$
\smash[b]{\bar{\xi }^{(\gamma )}(v)_\mathbf{n}=(w^{(d,\gamma )}\cdot v)_\mathbf{n}=\sum_{\mathbf{n}\in\mathbb{Z}^d} v_{\mathbf{n}-\mathbf{k}}w_\mathbf{k}^{(d,\gamma )}}
	$$
for every $v\in \mathcal{R}_\infty ^{(\gamma )}$, and by
	$$
\xi ^{(\gamma )}=\rho \circ \bar{\xi }^{(\gamma )}.
	$$

	\begin{prop}
	\label{codeprops}
The map $\xi^{(\gamma )}$ has the following properties.
	\begin{itemize}
	\item[(a)]
$\xi^{(\gamma )}(\mathcal{R}^{(\gamma )}_\infty) =X_{f^{(d,\gamma )}}$;
	\item[(b)]
For $v,v'\in\mathcal{R}^{(\gamma )}_\infty$, $\xi^{(\gamma )}(v)=\xi^{(\gamma )}(v')$ if and only if
	\begin{equation}
	\label{equivdis}
v'=v +f^{(d,\gamma )} \cdot h
	\end{equation}
for some $h\in \ell ^\infty(\mathbb{Z}^d,\mathbb{Z})$;
	\item[(c)]
$\xi^{(\gamma )}(v)\ne\xi^{(\gamma )}(v)$ for all $v,v\in \mathcal{R}^{(\gamma )}_\infty$ with $v-v'\in R_d$.
	\end{itemize}
Furthermore, the topological entropies of the shift-actions $\alpha _{f^{(d,\gamma )}}$ on $X_{f^{(d,\gamma )}}$ and $\sigma _{\mathcal{R}_\infty ^{(\gamma )}}$ on $\mathcal{R}_\infty ^{(\gamma )}$ coincide.
	\end{prop}

	\begin{proof}
The proofs are completely analogous to (but simpler than) those of the corresponding results in the critical case.
	\end{proof}

	\begin{coro}
	\label{c:correction}
For every $v\in\ell ^\infty (\mathbb{Z}^d,\mathbb{Z})$ there exists a $h\in\ell ^\infty (\mathbb{Z}^d,\mathbb{Z})$ such that $w=v+f^{(d,\gamma )}\cdot h\in\mathcal{R}_\infty ^{(\gamma )}$.
	\end{coro}

	\begin{proof}
This follows from Proposition \ref{codeprops} (a)--(b).
	\end{proof}

	\begin{rema}
	\label{r:correction}
The element $w$ in Corollary \ref{c:correction} can be constructed explicitly by using the method described in the proofs of Lemma \ref{l:surjective}, Theorem 4.1 and Subsection \ref{ss:group}.
	\end{rema}

In \cite{MRS}, two elements $v,v'\in \ell ^\infty(\mathbb{Z}^d,\mathbb{Z})$ are called equivalent (denoted by $v\sim v'$) if they satisfy \eqref{equivdis} for some $h\in \ell ^\infty(\mathbb{Z}^d,\mathbb{Z})$.\footnote{Definition 3.2 in \cite[page 404]{MRS} contains a misprint: the requirement that $h\in \ell ^\infty(\mathbb{Z}^d,\mathbb{Z})$ is omitted, although it is used subsequently.} We write $[v]\subset \ell ^\infty(\mathbb{Z}^d,\mathbb{Z})$ for the equivalence class of $v$ in this relation. The following theorem summarizes the results of \cite{MRS}.

	\begin{theo}
	\label{MRSresult}
The quotient $\mathcal{R}_\infty^{(\gamma )}/_\sim$ is a compact space. Moreover, $(\mathcal{R}^{(\gamma )}_\infty/_\sim,\oplus)$ is a compact abelian group, where
	$$
[y]\oplus[\tilde y]=[y+\tilde y].
	$$
Furthermore, there exists a shift-invariant measure of maximal entropy on $\mathcal{R}^{(\gamma )}_\infty$, denoted by $\mu $, such that
	\begin{equation}
	\label{largeset}
\mu\bigl(\bigl\{y\in \mathcal{R}_\infty^{(\gamma )}: [y]\cap \mathcal{R}_\infty^{(\gamma )}\enspace \textup{is a singleton}\bigr\}\bigr)=1.
	\end{equation}
	\end{theo}

	\begin{proof}
The first two statements are the results of \cite[Proposition 3.2 and Theorem 3.1]{MRS}. Furthermore, the main result of \cite{MRS}, Theorem 3.2, states that, if $\mu_{V}$ is the uniform measure on $\pi _{V}\bigl(\mathcal{R}^{(\gamma )}_{Q(N)}\bigr)$, where $V\subset \mathbb{Z}^d$ is a rectangle, than the set of limit points of sequences $\mu_{V}, V\nearrow \mathbb{Z}^d$, is a singleton. Denote by $\mu $ this unique limit point. We claim that $\mu $ is a shift-invariant measure on $\mathcal{R}^{(\gamma )}_\infty$, which moreover, has maximal entropy.

The invariance follows immediately from the uniqueness of the weak limit point. Denote by $\sigma$ the $\mathbb{Z}^d$-shift action on $\mathcal{R}^{(\gamma )}_\infty$. For every Borel set $A\subseteq \mathcal{R}^{(\gamma )}_\infty$, every $\mathbf{n}\in \mathbb{Z}^d$, and any sequence of rectangles $E _k\nearrow \mathbb{Z}^d$:
	$$
\mu( \sigma^\mathbf{n} A) =\lim_{k\to\infty} \mu_{E _k}( \sigma^\mathbf{n} A)=\lim_{k\to\infty} \mu_{E _k+\mathbf{n}}(A)=\mu(A).
	$$
Using the methods of \cite[Chapter 8]{Walters} (see also the proof of Theorem \ref{t:sand2} above), one can show that
	$$
h_\mu (\sigma _{\mathcal{R}^{(\gamma )}_\infty})=\lim_{E \to\mathbb{Z}^d} -\frac 1{|E |}\sum_{y_E \in \mathcal{R}^{(\gamma )}_E } \mu([y_E ])\log \mu([y_E ])= \lim_{E \to\mathbb{Z}^d} \frac 1{|E |}\log |\mathcal{R}^{(\gamma )}_E |=h_\textup{top}(\sigma _{\mathcal{R}^{(\gamma )}_\infty}),
	$$
where $\sigma _{\mathcal{R}^{(\gamma )}_\infty}$ is the restriction of $\sigma $ to $\mathcal{R}^{(\gamma )}_\infty$. Finally, (\ref{largeset}) is the result of \cite[Proposition 3.3]{MRS}.
	\end{proof}

We are now able to extend the results of \cite{MRS} further.
	\begin{theo}
	\label{t:unique}
Let $d\ge2$, $\gamma >2d$, and let $\mathcal{R}_\infty ^{(\gamma )}$ be the dissipative sandpile model \eqref{eq:sandpiles}.
	\begin{enumerate}
	\item[(i)]
The set $\mathcal{C}=\bigl\{y\in \mathcal{R}^{(\gamma )}_\infty: [y]\cap \mathcal{R}^{(\gamma )}_\infty\enspace \textup{is a singleton}\bigr\}$ is a dense $G_\delta$-subset of $\mathcal{R}^{(\gamma )}_\infty$;
	\item[(ii)]
The group $(\mathcal{R}^{(\gamma )}_\infty/_\sim,\oplus)$ is isomorphic to $X_{f ^{(d,\gamma)}}$;
	\item[(iii)]
The subshift $\mathcal{R}^{(\gamma )}_\infty$ admits a unique measure $\mu $ of maximal entropy.
	\item[(iv)]
The shift action of $\mathbb{Z}^d$ on $(\mathcal{R}^{(\gamma )}_\infty,\mu )$ is Bernoulli.
	\end{enumerate}
	\end{theo}

	\begin{proof}
The first statement is proved in Proposition \ref{p:injective2}. Using the properties of $\xi^\gamma:\mathcal{R}^{(\gamma )}_\infty\to X_{f ^{(d,\gamma)}}$ (Lemma \ref{codeprops}), the second statement is immediate. The same proof as in Theorem \ref{t:sand2} shows that $h_\textup{top}(\sigma _{\mathcal{R}^{(\gamma )}_\infty})= h_\textup{top}(X_{f^{(d,\gamma)}})$, and that $\xi ^{(\gamma )}_*\nu =\lambda _{X_{f^{(d,\gamma )}}}$ for every shift-invariant probability measure $\nu $ of maximal entropy on $\mathcal{R}_\infty ^{(\gamma )}$.

Since the restriction of the continuous map $\xi ^{(\gamma )}\colon \mathcal{R}_\infty ^{(\gamma )}\longrightarrow X_{f^{(d,\gamma )}}$ to $\mathcal{C}$ is injective, $\xi ^{(\gamma )}(\mathcal{C})$ is a Borel subset of $X_{f^{(d,\gamma )}}$ with full Haar measure.

If $\nu $ is a shift-invariant probability measure of maximal entropy on $\mathcal{R}_\infty ^{(\gamma )}$, then $\xi ^{(\gamma )}_*\nu =\lambda _{X_{f^{(d,\gamma )}}}$. Hence $\nu (\mathcal{C})=1$, and the injectiveness of $\xi ^{(\gamma )}$ on $\mathcal{C}$ implies that $\nu =\mu $, where $\mu $ is the measure appearing in Theorem \ref{MRSresult}. This proves (iii).

The Bernoulli property of the shift-action of $\mathbb{Z}^d$ on $(\mathcal{R}_\infty ^{(\gamma )},\mu )$ follows from the corresponding property of $\alpha _{f^{(d,\gamma )}}$ on $(X_{f^{(d,\gamma )}},\lambda _{X_{f^{(d,\gamma )}}})$ proved in \cite{RS}, since the two systems are measurably conjugate.
	\end{proof}

\section{\label{s:conclusions}Conclusions and final remarks}

(1) In \cite{Dhar3}, \textit{toppling invariants} have been constructed for the abelian sandpile model in finite volume. These are functions which are linear in height variables and are invariant under the topplings. It is also obvious that the definition \cite[Equation (3.3)]{Dhar3} cannot be extended to the infinite volume. The underlying problem (non-summability of the lattice potential function) is precisely the problem overcome by the introduction of $\ell^1$-homoclinic points $\{ v=g\cdot w^{(d)}:g\in I_d\}$. The inevitable drawback is a larger kernel $\xi_g\supsetneq f^{(d)}\cdot\ell^\infty(\Z^d,\Z)$. Nevertheless, we conjecture that for $d\ge 2$, the set $\{v\in \mathcal R_\infty:\textup{there exists}\enspace \tilde v\in\mathcal{R}_\infty: \tilde v\ne v \textup{ and } \xi_{I_d}(v)=\xi_{I_d}(\tilde v)\}$ has measure $0$ with respect to any measure of maximal entropy. As in the dissipative case, this would imply that $\mathcal{R}_\infty$ carries a unique measure of maximal entropy.

(2) In the present paper we did not address the properties of the infinite volume sandpile dynamics, see e.g. \cite{JR}. We note that the sandpile dynamics takes a particularly simple form in the image space, the harmonic model $X_{f^{(d)}}$ or its factor group $\tilde{X}_{f^{(d)}}$. Namely, given any initial configuration $v$, suppose one grain of sand is added at site $\n$. For every $g\in \tilde{I}_d=I_d\smallsetminus (f^{(d)})$,
	$$
\xi_g(v+\delta ^{(\mathbf{n})}) =\xi _g(v)+\xi _g(\delta ^{(\mathbf{n})})= \xi _g(v)+\rho (\alpha^{-\n}z^{(g)}),
	$$
where $\delta ^{(\mathbf{n})}=\sigma ^{-\mathbf{n}}\delta ^{(\mathbf{0})}$ (cf. Footnote \ref{delta}) and $z^{(g)}=\rho (g^*\cdot w^{(d)})\in\Delta _\alpha ^{(1)}(X_{f^{(d)}})$ is the homoclinic point appearing in \eqref{eq:Delta}. It might be interesting to understand whether any statistical properties of the harmonic model can be used to draw any conclusions on the distribution of avalanches and other dynamically relevant notions in $\mathcal R_\infty$.

Finally, as already mentioned in the introduction, the group $\mathcal{G}_d=R_d/(f^{(d)})$ is the appropriate infinite analogue of the groups of addition operators in finite volumes: on the sandpile model, $\mathcal{G}_d$ can be viewed as the abelian group generated by the elementary addition operators $\{a_\n : \n\in \Z^d\}$ satisfying the basic relations
	$$
a_\n^{2d} =\prod_{\k:\|\k-\n\|_{\max}=1} a_\k
	$$
for all $\n\in\Z^d$. These addition operators are well-defined on $\mathcal{R}_E$, $E\Subset\mathbb{Z}^d$, but for the infinite volume limit $\mathcal{R}_\infty $ these operators are not defined everywhere. Under the maps $\xi _g\colon \mathcal{R}_\infty \longrightarrow X_{f^{(d)}},\,\,g\in I_d$, or $\xi _{I_d}\colon \mathcal{R}_\infty \longrightarrow \tilde{X}_{f^{(d)}}=X_{f^{(d)}}/X_{I_d}$, the addition operator $a_\mathbf{n}$ is sent to addition of the homoclinic points $\xi _g(\delta ^{(\mathbf{n})})=\rho (g^*\cdot w^{(d)})=g(\alpha )(x^\Delta )$ (on $X_{f^{(d)}}$) and $\xi _{I_d}(\delta ^{(\mathbf{n})})$ (on $\tilde{X}_{f^{(d)}}$), respectively. These additions are defined everywhere on $X_{f^{(d)}}$ and $\tilde{X}_{f^{(d)}}$, and the isomorphism between $\tilde{X}_{f^{(d)}}$ and $\mathcal{R}_\infty /_\sim$ implies that \textit{the addition operators $a_\mathbf{n},\,\mathbf{n}\in\mathbb{Z}^d$, are defined everywhere} on $\mathcal{R}_\infty /_\sim$ (cf. Subsection \ref{ss:group}).


\begin{thebibliography}
{99}

\bibitem{Jarai1} S. R. Athreya and A. A. J{\'a}rai, \textit{Infinite volume limit for the stationary distribution of abelian sandpile models}, Comm. Math. Phys. {\bf 249} (2004), no. 1, 197--213.

\bibitem{Jarai2} S. R. Athreya and A. A. J{\'a}rai, \textit{Erratum: Infinite volume limit for the stationary distribution of abelian sandpile models}, Comm. Math. Phys. {\bf 264} (2006), no. 3, 843.


\bibitem{BTW1} P. Bak, C. Tang and K. Wiesenfeld, \textit{Self-organized criticality: An explanation of the $1/f$ noise}, Phys. Rev. Lett. \textbf{59} (1987), 381--384.

\bibitem{BTW2} P. Bak, C. Tang and K. Wiesenfeld, \textit{Self-organized criticality}, Phys. Rev. A \textbf{38} (1988), 364--374.

\bibitem{BP} R. Burton and R. Pemantle, \textit{Local characteristics, entropy and limit theorems for spanning trees and domino tilings via transfer-impedances}, Ann. Probab. \textbf{21} (1993), 1329--1371.

\bibitem{deBHR} C. de Boor, K. H\"{o}llig and S. Riemenschneider, \textit{Fundamental solutions for multivariate difference equations}, Amer. J. Math. \textbf{111} (1989), 403--415.

\bibitem{dis2} F. Daerden and C. Vanderzande, \textit{Dissipative abelian sandpiles and random walks}, Phys. Rev. E \textbf{63} (2001), 30301--30304.

\bibitem{Dhar0} D. Dhar, \textit{Self Organized Critical State of Sandpile Automaton Models}, Phys. Rev. Lett. \textbf{64} (1990), 1613--1616.

\bibitem{Dhar1} D. Dhar, \textit{The abelian sandpiles and related models}, Phys. A \textbf{263} (1999), 4--25.

\bibitem{Dhar2} D. Dhar, \textit{Theoretical studies of self-organized criticality}, Phys. A \textbf{369} (2006), 29--70.

\bibitem{Dhar3} D. Dhar, P. Ruelle, S. Sen, and D.-N.Verma, \textit{Algebraic aspects of abelian sandpile models}, {J. Phys. A} \textbf{28} (1995),805-831.




\bibitem{FU} Y. Fukai and K. Uchiyama, \textit{Potential kernel for the two-dimensional random walk}, Ann. Probab. \textbf{24} (1996), 1979--1992.

\bibitem{JR} A. J\'{a}rai and F. Redig, \textit{Infinite volume limit of the Abelian sandpile model in dimensions $d \ge 3$}, Probab. Theor. Relat. Fields \textbf{141} (2008), 181--212.

\bibitem{LS} D. Lind and K. Schmidt, \textit{Homoclinic points of algebraic $\mathbf{Z}^d$-actions}, J. Amer. Math. Soc. \textbf{12} (1999), 953--980.

\bibitem{LSW} D. Lind, K. Schmidt and T. Ward, \textit{Mahler measure and entropy for commuting automorphisms of compact groups}, Invent. Math. \textbf{101} (1990), 593--629.

\bibitem{MRS} C. Maes, F. Redig and E. Saada, \textit{The infinite volume limit of dissipative abelian sandpiles}, Comm. Math. Phys. \textbf{244} (2004), 395--417.

\bibitem{Pemantle} R. Pemantle, \textit{Choosing a spanning tree for the integer lattice uniformly}, Ann. Probab. \textbf{19} (1991), 1559--1574.

\bibitem{vdPT} M. van der Put, F.L. Tsang, \textit{Discrete Systems and Abelian Sandpiles}, preprint (2008).

\bibitem{Redig} F. Redig, \textit{Mathematical aspects of the abelian sandpile model}, Mathematical Statistical Physics, Volume Session LXXXIII: Lecture Notes of the Les Houches Summer School 2005 (Les Houches), A. Bovier, F. Dunlop, F. den Hollander, A. van Enter and J. Dalibard (eds.), Elsevier, (2006), pp. 657-728.

\bibitem{RS} D.J. Rudolph and K. Schmidt, \textit{Almost block independence and Bernoullicity of $\mathbb Z^d$-actions by automorphisms of compact groups}, Invent. Math. \textbf{120} (1995), 455--488.

\bibitem{DSAO} K. Schmidt, \textit{Dynamical Systems of Algebraic Origin}, Birkh\"auser Verlag, Basel-Berlin-Boston, 1995.

\bibitem{Sheffield} S. Sheffield, \textit{Uniqueness of maximal entropy measure on essential spanning forests}, Ann. Probab. \textbf{34} (2006), 857--864.

\bibitem{Sinai} Ya.G. Sinai, \textit{On a weak isomorphism of transformations with invariant measure}, Mat. Sb. \textbf{63 (105)} (1964), 23--42.

\bibitem{Solomyak} R. Solomyak, \textit{On coincidence of entropies for two classes of dynamical systems}, Ergod. Th. \& Dynam. Sys. \textbf{18} (1998), 731--738.

\bibitem{S} F. Spitzer, \textit{Principles of random walks}, van Nostrand Reinhold, New York, 1964.

\bibitem{dis1} T. Tsuchiya and M. Tomori, \textit{Proof of breaking of self-organized criticality in a nonconservative abelian sandpile model}, Phys. Rev. E \textbf{61} (2000), 1183--1188.

\bibitem{U} K. Uchiyama, \textit{Green's function for random walks on $\mathbf{Z}^N$}, Proc. London Math. Soc. \textbf{77} (1998), 215--240.

\bibitem{Walters} P. Walters, \textit{An introduction to ergodic theory}, Graduate Texts in Mathematics, vol. 79, Springer Verlag, Berlin-Heidelberg-New York, 1982.

	\end{thebibliography}
	\end{document}